\documentclass[10pt]{article}
\usepackage{indentfirst, latexsym, amsmath,bm,amssymb}
\usepackage{tikz}
\usepackage{color,soul}
\usepackage{amsmath}
\usepackage{amssymb}
\usepackage{mathrsfs}
\usepackage{mathtools}
\usepackage{stmaryrd}
\usepackage{bbold}

\allowdisplaybreaks[4]

\begin{document}
\newtheorem{claim}{Claim}
\newtheorem{fact}{Fact}
\newtheorem{theorem}{Theorem}[section]
\newtheorem{corollary}[theorem]{Corollary}
\newtheorem{definition}[theorem]{Definition}
\newtheorem{conjecture}[theorem]{Conjecture}
\newtheorem{question}[theorem]{Question}
\newtheorem{lemma}[theorem]{Lemma}
\newtheorem{remark}[theorem]{Remark}
\newtheorem{proposition}[theorem]{Proposition}
\newtheorem{example}[theorem]{Example}
\newenvironment{proof}{\noindent {\bf
Proof.}}{\rule{3mm}{3mm}\par\medskip}
\newenvironment{proofof}{\noindent {\bf
Proof of  Theorem}}{\rule{2mm}{2mm}\par\medskip}
\newcommand{\pp}{{\it p.}}
\newcommand{\de}{\em}

\title{  {Tur\'{a}n extremal graphs vs. Signless Laplacian spectral Tur\'{a}n extremal graphs
}\thanks{
This work is  supported by NSFC (Nos. 12461063, 12101166, 12371349, 12371354, 12471331), NSFC-BRFFR (No. 12311530761), Hainan Provincial Natural Science Foundation
of China (No. 123MS005) %Shanghai University Young Faculty Grant  %(No. B3A010023045014)
and
Youth Foundation of Shanghai University of International Business and Economics (No. 24QN010).
The authors   are listed in alphabetical order.
E-mail addresses: mzchen@hainanu.edu.cn (M.-Z. Chen), yaleijin@shnu.edu.cn ($^\dag$Y.-L. Jin, corresponding author), zpengli0626@163.com ($^\dag$P.-L. Zhang, corresponding author), xiaodong@sjtu.edu.cn($^\dag$X.-D. Zhang).}}
 %fenglh@163.com (L. Feng),  wjliu6210@126.com(W. Liu).
\date{}
\author{
Ming-Zhu Chen$^b$, Ya-Lei Jin$^a\dag$, Peng-Li Zhang$^c\dag,$ Xiao-Dong Zhang$^d$\\
{\small $^a$ Department of Mathematics, Shanghai Normal University, Shanghai 200234, P.R. China
} \\
{\small  $^b$ School of Mathematics and Statistics, Hainan University,
Haikou 570228, P.R. China}\\
{\small $^c$ School of Statistics and Data Science, Shanghai University of International Business and} \\
{\small   Economics, Shanghai 201620, P.R. China}\\
{\small  $^d$ School of Mathematical Sciences,  MOE-LSC, SHL-MAC, Shanghai Jiao Tong University,}\\ {\small  Shanghai 200240, PR China}
}

\maketitle

\vspace{-0.5cm}

\begin{abstract}
%Let $F$ be a graph with chromatic number $\chi(F)=r+1$ and $Ex(n,F)$ be the set of $n$-vertex $F$-free graphs with maximum number of edges, where a graph is called $F$-free if it does not contain $F$ as a  subgraph.  Let   $1 \gg\epsilon >0,$ $r\geqslant 2$ and $\mathcal{G}_{n}$ be the collection of all $n$-vertex $F$-free graphs with
%minimum degree more than $(1-\frac{1}{r}-\epsilon)n.$
%where the chromatic number of $F$ is $r+1$.
%Employing Szemer\'{e}di's graph regularity lemma and F\"{u}redi stability theorem, we  characterize the structure of graphs in $\mathcal{G}_{n}$. As its applications, if $ex(n,F)=t_r(n)+O(1),$  $n$ is sufficiently large and $r\geqslant 3$, then $Ex_{ssp}(n,F)\subseteq Ex(n,F)$, where $Ex_{ssp}(n,F)$ is the set of $n$-vertex $F$-free  graphs with maximum signless Laplacian spectral radius.

Let $F$ be a graph with chromatic number $\chi(F) = r+1$. Denote by $ex(n, F)$  and $Ex(n, F)$  the Tur\'{a}n number and  the set of all extremal graphs for $F$, respectively. In addition,  $ex_{ssp}(n, F)$  and $Ex_{ssp}(n, F)$ are the maximum signless Laplacian spectral radius of all $n$-vertex  $F$-free graphs and
the set of all $n$-vertex  $F$-free graphs  with signless Laplacian spectral radius  $ex_{ssp}(n, F)$, respectively.  It is known that $Ex_{ssp}(n, F)\supset Ex(n, F)$ if $F$ is a triangle. In this paper, employing %Szemer\'{e}di's graph
the regularity method and F\"{u}redi's stability theorem,
we prove that for a given graph $F$ and $r\geqslant 3$, if $ex(n, F) = t_r(n)+O(1)$, then $ Ex_{ssp}(n, F) \subseteq Ex(n, F)$ for sufficiently large $n$, where $t_r(n)$ is the number of  edges in  the Tur\'{a}n graph $T_r(n)$.
\end{abstract}

{AMS Classification: 05C50, 05C35}

{{\bf Keywords:}
Regularity method;
%Szemer\'{e}di's graph regularity lemma;
F\"{u}redi's stability theorem;
 Tur\'{an} number; Spectral radius; Extremal graph}
\section{Introduction}
%Denote by  $T_{r}(n)$  the {\it Tur\'{a}n graph}, which is .
Given a graph $F$, a graph is called $F$-free if it does not contain $F$ as a subgraph.
The classic Tur\'{a}n-type problem asks what is the maximum number of edges in  an $n$-vertex $F$-free  graph, where the maximum number of edges is called the {\it Tur\'{a}n number} of $F$,  denoted by $ex(n,F).$ The $n$-vertex $F$-free  graph with $ex(n,F)$ edges is called the {\it extremal graph} for $F$ and the set of all extremal graphs on $n$ vertices for $F$  is denoted by $Ex(n,F).$
In 1941, Tur\'{a}n~\cite{Turan 1941} proved  Tur\'{a}n theorem, which states that the Tur\'{a}n graph $ T_r(n),$ the balanced complete $r$-partite graph on $n$ vertices, is the unique extremal graph for $K_{r+1}$.
Denote by $t_{r}(n)$ the number of edges in $ T_r(n).$
The celebrated Erd\H{o}s-Stone-Simonovits theorem \cite{EM1966,ES1946}
states $$ex(n, F) = \bigg(1 - \frac{1}{r}\bigg)\frac{n^2}{2} +o(n^2),$$ where $\chi(F)=r+1$  is the chromatic number of $F$.
%There are many results on  Tur\'{a}n-type problems, see~\cite{}.
A graph is color-critical if it contains an edge whose deletion reduces its chromatic number. %There are many color-critical graphs, for example, complete  graphs, odd cycles and even wheels.
For any color-critical graph $F$ with chromatic number $\chi(F)=r+1$, Erd\H{o}s-Stone-Simonovits theorem shows that $ex(n, F)=(1-\frac{1}{r})\frac{n^2}{2} +o(n^2).$ Later, Simonovits~\cite{Simonovits1968} gave the exact Tur\'{a}n number of $F$ and proved that  $T_r(n)$ is the unique extremal graph for large $n$. %which is called chromatic critical edge theorem  by F\"{u}redi and Gunderson \cite{FG2015}.

\begin{theorem} [Simonovits, \cite{Simonovits1968}]\label{thm1} Let $F$ be a color-critical graph with $\chi(F)=r+1$. Then
there exists a number $n_0(F)$ such that $T_r(n)$ is the only extremal graph with respect to
$ex(n, F)$ provided $n \geqslant  n_0(F)$.
\end{theorem}

Let $A(G)$ be the adjacency matrix  and $Q(G)=D(G)+A(G)$ be  the signless Laplacian matrix of a graph $G$, respectively,
where $D(G)$ is the degree diagonal matrix of $G$. The largest eigenvalue of  $A(G)$ (resp. $Q(G)$) is called the  {\it adjacency}  (resp. {\it signless Laplacian}) {\it spectral radius} of $G$,  denoted by $\lambda(G)$ (resp. $q(G)$).
The {\it signless Laplacian} {\it spectral Tur\'{a}n number} of a graph $F$, denoted by $ex_{ssp}(n,F)$, is the maximum  signless Laplacian spectral radius   among all $n$-vertex $F$-free graphs.
Denote by $Ex_{ssp}(n,F)$ the set of all $n$-vertex $F$-free graphs with  the signless Laplacian spectral radius $ex_{ssp}(n,F)$.

%The {\it adjacency} (resp. {\it signless Laplacian}) {\it spectral Tur\'{a}n number} of a graph $F$, denoted by $ex_{sp}(n,F)$ (resp. $ex_{ssp}(n,F)$), is the largest  adjacency (resp. signless Laplacian) spectral radius   of an $n$-vertex $F$-free graph.
%Denote by $Ex_{sp}(n,F)$ (resp. $Ex_{ssp}(n,F)$) be the set of $F$-free graphs with the number of edges $ex_{sp}(n,F)$ (resp. $ex_{ssp}(n,F)$).
In 2010, Nikiforov \cite{Nikiforov2010} proposed the adjacency spectral  Tur\'{a}n-type problem which asks  %for
the maximum  adjacency spectral radius of an $n$-vertex $F$-free graph.
Guiduli \cite{Guiduli1996} and Nikiforov \cite{Nikiforov2007} independently proved adjacency spectral Tur\'{a}n theorem, showing that if $G$ attains the maximum adjacency spectral radius among all $n$-vertex $K_{r+1}$-free  graphs, then $G=T_r(n).$ Nikiforov \cite{Nikiforov2011} proved the adjacency spectral Erd\H{o}s-Stone-Simonovits theorem.
%Denote $(k, t)$-fan by  $F_{k,t},$ the
A graph on $(t-1)k + 1$ vertices consisting of $k$ cliques each with $t$ vertices which
intersect in exactly one common vertex is called a {\it $(k, t)$-fan} and denoted by  $F_{k,t}$.
When $t=3$, $F_{k,3}$ is also called  the  {\it $k$-fan} and denoted by $F_{k}.$
Cioab\v{a}, et al.~\cite{Cioaba Feng Tait Zhang 2020} showed that if $F$ is a $k$-fan $F_{k}$ and $G$ has the maximum adjacency spectral radius among all $n$-vertex $F$-free  graphs, then $G\in  Ex(n,F)$ for sufficiently large $n.$
Desai, et al.~\cite{DKLNTW2022} showed that for $k\geqslant 2$ and  $t\geqslant3,$ if $F$ is a $(k,t)$-fan $F_{k,t}$  and $G$ attains the maximum adjacency spectral radius among all $n$-vertex $F$-free  graphs, then $G\in  Ex(n,F)$ for sufficiently large $n.$
Cioab\v{a}, Desai and Tait~\cite{Cioaba Desai Tait 2022} investigated  the maximum adjacency spectral radius  of an $n$-vertex graph containing no odd wheel  and  posed the following conjecture: For any graph $F$  satisfying that the graphs in $Ex(n,F)$ are Tur\'{a}n graphs plus $O(1)$ edges,
if $G$ attains the maximum adjacency spectral radius among all $n$-vertex $F$-free graphs,  then $G\in  Ex(n,F)$ for  sufficiently large $n.$
%\begin{conjecture}[\cite{Cioaba Desai Tait 2022}]\label{conjecture cioaba}
%Let $F$ be any graph such that the graphs in $Ex(n,F)$ are Tur\'{a}n graphs plus $O(1)$ edges. For
%\end{conjecture}
In 2022, Wang, Kang  and Xue~\cite{Wang Kang Xue 2023}  confirmed this conjecture %~\ref{conjecture cioaba}
 and presented the following stronger result.
\begin{theorem}[\cite{Wang Kang Xue 2023}]\label{theorem of Wang Kang Xue}
Let $r\geqslant 2$ be an integer and $F$ be a graph with $ex(n,F)=t_{r}(n)+O(1)$. For
sufficiently large $n,$  if $G$ attains the maximum adjacency spectral radius among all $n$-vertex $F$-free graphs,  then $G\in  Ex(n,F).$
\end{theorem}

A similar problem  is that  whether the signless Laplacian spectral analogue of Theorem~\ref{theorem of Wang Kang Xue} holds.
However, there exist some counterexamples for $r=2.$ Let $k\geqslant 2.$ A book  is a graph consisting of $k$ triangles sharing one common edge, which is denoted by  $B_{k}.$ When $F$ is a triangle $K_{3},$ an odd cycle $C_{2k+1}$ or a book $B_{k},$ %since $\chi(F)=3$
  $Ex(n,F)=\{T_{2}(n)\}$ for large $n$ by Theorem~\ref{thm1}.
%$Ex(n,K_{3})=\{T_{2}(n)\}$~\cite{Turan 1941} and
%In 2013, He, Jin and Zhang~\cite{He Jin Zhang 2013} proved that all complete bipartite graphs $K_{s,n-s}$ on $n$ vertices are the extremal graphs with respect to $ex_{ssp}(n,K_{3}),$ where $1\leqslant  s  \leqslant  n-1.$
Clearly, $Ex_{ssp}(n, K_{3})\supset Ex(n, K_{3})$ since  $Ex_{ssp}(n,K_{3})=\{K_{s,n-s}:1\leqslant  s  \leqslant  n-1\}$~\cite{He Jin Zhang 2013}.
%which reveals that Theorem~\ref{theorem of Wang Kang Xue} does not hold for signless Laplacian spectral radius if $r=2.$
%By Tur\'{a}n theorem~\cite{Turan 1941},
%we have $Ex(n,K_{3})=\{T_{2}(n)\}.$
%By~\cite{He Jin Zhang 2013}, we have $Ex_{ssp}(n,K_{3})=\{K_{s,n-s}:1\leqslant  s  \leqslant  n-1\},$ which implies that
%In 2013, He, Jin and Zhang~\cite{He Jin Zhang 2013} proved that all complete bipartite graphs $K_{s,n-s}$ on $n$ vertices are the extremal graphs with respect to $ex_{ssp}(n,K_{3}),$ where $1\leqslant  s  \leqslant  n-1.$
%Clearly, $Ex_{ssp}(n, F)\supset Ex(n, F).$
%which reveals that Theorem~\ref{theorem of Wang Kang Xue} does not hold for signless Laplacian spectral radius if $r=2.$
And $Ex_{ssp}(n, C_{2k+1}) \cap Ex(n, C_{2k+1})=\emptyset$ because  $Ex_{ssp}(n,C_{2k+1})=\{K_{k}\vee \overline{K}_{n-k}\}$~\cite{Yuan2014} for large $n$, where $``\vee"$ is the join operation. Also,
$Ex_{ssp}(n, B_{k}) \cap Ex(n, B_{k})=\emptyset$ for large $n$ by Theorem~1.4 in~\cite{Chen Jin Zhang 2025}.
Besides,
for large $n,$ $Ex(n, F_{k})=\{G_{n,k}^{1}\}$ if  $k$ is odd and
$Ex(n, F_{k})=\{G_{n,k}^{2}\}$ if  $k$ is even~\cite{Erdos Furedi Gould Gunderson 1995},
%For $n\geqslant 50k^2,$  Erd\H{o}s, et al.~\cite{Erdos Furedi Gould Gunderson 1995} %F\"{u}redi, Gould and Gunderson
%proved that the graphs $G_{n,k}^{i}(i=1,2)$ are the only extremal graphs with respect to $ex(n, F_{k}),$
where
$G_{n,k}^{1}$ is the graph obtained by taking
%a complete equi-bipartite graph
the  Tur\'{a}n graph $T_{2}(n)$
and embedding two vertex disjoint copies of $K_{k}$ on one side if $k$ is odd (where $n\geqslant 4k-1$), and $G_{n,k}^{2}$ is the graph obtained by taking
%a complete equi-bipartite graph
the Tur\'{a}n graph $T_{2}(n)$
and embedding
a graph with $2k-1$ vertices, $k^2-\frac{3}{2}k$ edges with  maximum degree $k-1$ in one side
%otherwise.
if $k$ is even (where $n\geqslant 4k-3$).
 Clearly,  $Ex_{ssp}(n,F_{k})\cap Ex(n,F_{k})=\emptyset $ since %$n\geqslant  3k^2-k-2,$
 $Ex_{ssp}(n,F_{k})=\{K_{k}\vee \overline{K}_{n-k}\}$~\cite{Zhao Huang Guo2021}  for large $n$.

Let $r\geqslant 3.$
According to   $Ex(n,K_{r+1})=\{T_{r}(n)\}$~\cite{Turan 1941} and $Ex_{ssp}(n,K_{r+1})=\{T_r(n)\}$~\cite{He Jin Zhang 2013},
$Ex_{ssp}(n, K_{r+1})= Ex(n, K_{r+1}).$
%Furthermore,
%He, Jin and Zhang~\cite{He Jin Zhang 2013} proved that
%the Tur\'{a}n graph $ T_r(n)$ is the unique extremal graph with respect to $ex_{ssp}(n,K_{r+1}).$  %for $r\geqslant 3.$
%Chen, Jin and Zhang~\cite{Chen Jin Zhang 2025} determined that the extremal graphs attaining the largest signless Laplacian spectral radius among all book-free graphs of order $n.$
%Let $F$ be a color-critical graph $F$ with $\chi (F)=r+1\geqslant 3$ For $$ Simonovits~\cite{} proved
Furthermore, for any color-critical graph $F$ with chromatic number $r+1$ %$\geqslant 4$
and sufficiently large $n,$
since $Ex(n,F)=\{T_{r}(n)\}$ %for sufficiently large $n$
by Theorem~\ref{thm1} and $Ex_{ssp}(n,F)=\{T_{r}(n)\}$~\cite{Zheng Li Li 2025}, $Ex_{ssp}(n, F)= Ex(n, F)$
 %Zheng, Li and Li
 for sufficiently large $n$.
%showed that the Tur\'{a}n graph $T_{r}(n)$ is the unique extremal graph with respect to $ex_{ssp}(n,F).$
%where $F$ is  any and  if  $G$ is an $n$-vertex $F$-free graph, % of order $n,$
%then $q(G)\leqslant  q(T_{r}(n)),$ with equality if and only if $G=T_{r}(n).$
The above two results indicate that Theorem~\ref{theorem of Wang Kang Xue} may hold for signless Laplacian spectral radius for $r\geqslant3.$

Very recently,
Zheng, Li and Su~\cite{Zheng Li Su 2025} proved a signless Laplacian spectral
version of the Erd\"{o}s-Stone-Simonovits theorem.
Besides,
Desai, et al.~\cite{DKLNTW2022}  proposed a problem as follows: for $k\geqslant 1$ and $t\geqslant 3,$ %if $t\geqslant  3,$
if $n$ is sufficiently large and $G$ has the maximum signless Laplacian spectral radius among all $n$-vertex $F_{k,t}$-free graphs, then $G$ is the complete split graph $K_{k(t-2)}\vee \overline{K}_{n-k(t-2)}.$
%$S_{n,k(t-2)}$ is the join of  a clique on $k(t-2)$ vertices and an independent set on $n-k(t-2)$ vertices.
For more results about the  spectral Tur\'{a}n-type problems, see~\cite{Byrne 2025,Byrne Desai Tait 2024,Fang Tait Zhai 2025,Li Feng Peng 2024,Li Feng Peng 2025,Lou Lu Zhai 2025,Ni Wang Kang 2023,Zheng Li Fan 2025}.
Motivated by the above results,
it is natural to ask  the following general question:
% for $r\geqslant 3$.

\begin{question}\label{question}
Let $r\geqslant 3$ be an integer and $F$ be a graph with $ex(n,F)=t_{r}(n)+O(1)$. Is it true that $Ex_{ssp}(n,F)\subseteq Ex(n,F)$ for sufficiently large $n?$
\end{question}

In this paper, by utilizing
the regularity method %Szemer\'{e}di's graph regularity lemma
 and F\"{u}redi's stability theorem,
%a novel technique,
%Szemer\'{e}di's graph regularity lemma and F\"{u}redi stability theorem,
%two novel and powerful tools,  % technique
%  rather than %~\cite{Wang Kang Xue 2023} the spectral version of the Stability Lemma
%which is totally different from the technique used in proving Theorem~\ref{theorem of Wang Kang Xue},
we provide a positive answer to Question~\ref{question}, %for $r\geqslant 3$,
%which can be seen as an extension of %conjecture
%~\ref{conjecture cioaba} and %the signless Lapalcian spectral version of
%Theorem~\ref{theorem of Wang Kang Xue}.
%the signless Laplacian spectral analogue of Theorem~\ref{theorem of Wang Kang Xue} as follows.

\begin{theorem}\label{main result}
Let $r\geqslant 3$ be an  integer and $F$ be any graph such that  $ex(n,F)=t_{r}(n)+O(1).$
%are obtained from $T_{r}(n)$ by adding $O(1)$ edges.
 %if $G$ has the maximal signless Laplacian spectral radius over all $n$-vertex $F$-free graphs, then
Then $Ex_{ssp}(n,F)\subseteq Ex(n,F)$  for sufficiently large $n.$
\end{theorem}

The remainder of this paper is organized as follows. In Section~2, we provide some necessary notations and auxiliary tools. In Section~3, for $1 \gg\epsilon >0$ and a graph $F$  with $ex(n,F)=t_{r}(n)+O(1),$ %chromatic number  $\chi(F)=r+1\geqslant 3$,
 employing the regularity method %Szemer\'{e}di's graph regularity lemma
 and F\"{u}redi's stability theorem, we  characterize the structure of graphs in $\mathcal{G}_{n}$, where $\mathcal{G}_{n}$ is the collection of all $n$-vertex $F$-free graphs with
minimum degree more than $(1-\frac{1}{r}-\epsilon)n.$ As its applications,
in Section~4, we present the proof of Theorem~\ref{main result} and provide a negative answer to the problem proposed by Desai, et al.~\cite{DKLNTW2022}.

\section{Preliminaries}
We provide some notations which will be used in this paper.
Let $G$ be a simple graph with  vertex set $V(G)$   and  edge set $E(G).$
We call $|V(G)|$ the {\it order} of $G.$ Denote by $e(G)=|E(G)|.$
For two vertices $u,v\in V(G),$ $u$ and $v$ are {\it adjacent} if $uv\in E(G),$ which is denoted by $u\sim v.$
For a vertex $u\in V(G),$ the set $\{v\in V(G)\ | \ u\sim v\}$ is  the {\it neighborhood}  of $u$ in $G,$ which is denoted by $N_{G}(u).$
We call $|N_{G}(u)|$   the {\it degree} of a vertex $u$ in $G,$  which is denoted by $d_{G}(u).$
For brevity, sometimes we  omit  $G$ from  these notations if it is clear from the context.
For two graphs $H$ and  $G,$   $H\subseteq G$ denotes $H$ is a subgraph of $G.$
For two disjoint vertex sets $A,B\subseteq V(G),$  $E(A)$ denotes the subset of edges of $G$ whose end vertices are both in $A,$
$E(A,B)$ denotes the subset of edges of $G$ with one end vertex in $A$ and the other  in $B,$ and $\overline{E(A,B)}$ denotes the set of all missing edges  between $A$ and $B.$ %but not in $E(A,B).$
Denote by $e(A)=|E(A)|$ and $e(A,B)=|E(A,B)|.$
The graph $G[A]$ is called  the {\it induced subgraph} by $A,$ which is a subgraph of $G$ with vertex set $A$ and edge set $E(A).$
Especially, for $u\in V(G),$ $G-u$ denotes the induced subgraph  by $V(G)\setminus\{u\}$
and $d_{A}(u)$ denotes the number of vertices in $A$ adjacent to  $u.$
For two graphs $G$ and $H$ with the same vertex set $V(G)$,  denote by  $G\cup H$ the graph with vertex set $V(G)$ and edge set $E(G)\cup E(H).$
For two vertex disjoint graphs $G$ and $H,$ denote
by $G\vee H$ the graph obtained from $G$ and $H$ by adding all edges between  $V(G)$ and $V(H).$
Let $K_{n_{1},\cdots,n_{r}}$ be the complete $r$-partite graph with parts of sizes $n_{1},\cdots,n_{r}.$
For a graph $G,$ $\overline{G}$ denotes the complement of $G.$
Readers are referred to the book~\cite{Bondy Murty2008} for undefined notations and definitions.

%$A\mathbf{x}=\rho(A)\mathbf{x}$ and
%its maximum component is $1.$
 %satisfying that $A\mathbf{x}=\rho(A)\mathbf{x}.$
%A nonnegative vector
%$\mathbf{x}$ is called {\it Perron vector} if it satisfies $A\mathbf{x}=\rho(A)\mathbf{x}$ and
%its maximum component is $1.$

Let $G$ be a graph and $U,W\subseteq V(G)$. %Denote by $e(U,W)$ the number of edges between $U$ and $W.$
Define the {\it edge density} between $U$ and $W$ in $G$ by $d(U,W)=\frac{e(U,W)}{|U||W|}.$
We call $(U,W)$ an {\it $\epsilon$-regular pair} in $G$ if for all $A\subseteq U$ and $B\subseteq W$ with $|A|\geqslant\epsilon\left|U\right|$ and $|B|\geqslant\epsilon\left|W\right|$, one has
$|d(A,B)-d(U,W)|\leqslant \epsilon.$
A partition $\mathcal{P}=\{V_{1},\cdots,V_{k}\}$ of its vertex set is called {\it an $\epsilon$-regular partition} if
$$\sum_{\substack{(i,j) \in [k]^{2} \\ (V_{i},V_{j}) \text{ not } \epsilon \text{-regular}}}|V_{i}||V_{j}| \leqslant  \epsilon|V(G)|^{2}.$$

Szemer\'{e}di's graph regularity lemma~\cite{Szemerdi1978}  is one of %and graph counting lemma~\cite{}  are
the most powerful tools in extremal graph theory.
There are many important applications of the regularity lemma, we refer the reader
to nice surveys~\cite{Komlos Shokoufandeh  Simonovits Szemeredi 2002,Komlos Simonovits 1996,Rodl Schacht 2010} and other related references.  Many applications of the regularity lemma, including the triangle removal lemma~\cite{Erdos Furedi Gould Gunderson 1995,Fox 2011,Ruzsa  Szemeredi 1978}, rely on also having an associated {\it graph counting lemma}.
%for example, see the surveys for more information~\cite{Komlos Shokoufandeh  Simonovits Szemeredi 2002,Komlos Simonovits 1996}.
This combined application of the regularity lemma and a counting lemma is often referred
to as the {\it regularity method}.
Here we adopt the versions stated by Zhao~\cite[Theorem~2.1.9]{Zhao 2023} and ~\cite[Theorem~2.6.4]{Zhao 2023}.

\begin{theorem}[Szemer\'{e}di's graph regularity lemma, \cite{Zhao 2023}]
For every $\epsilon > 0$, there exists a constant $M$ such that every graph has an $\epsilon$-regular partition into at most $M$ parts.
\end{theorem}

\begin{theorem}[Graph counting lemma, \cite{Zhao 2023}]\label{graph counting lemma}
Let $F$ be a graph with maximum degree $\Delta(F)\geqslant1,$ $\epsilon>0$ and $G$ be a graph. Let $X_{i}\subseteq V(G)$ for each $i\in V(F).$ Suppose that for each $ij\in E(F),$ $(X_{i},X_{j})$ is an $\epsilon$-regular pair with edge density $d(X_{i},X_{j})\geqslant  (\Delta(F)+1)\epsilon^{1/\Delta(F)}.$ If $|X_{i}|\geqslant |V(F)|/\epsilon$ for each $i,$ then there exists such a homomorphism $F\rightarrow G$ that is injective (i.e., an embedding of $F$ as a subgraph).
\end{theorem}

\begin{lemma}[\cite{Furedi2015}]\label{Furedi}
    Suppose that $K_{p+1} \nsubseteq G$, $|V(G)| = n$, $t \geqslant 0$, and
$e(G) = t_{p}(n) -t.$
Then there exists an (at most) $p$-chromatic subgraph $H_0$, $E(H_0) \subset E(G)$ such that
$$
e(H_0) \geqslant e(G) - t.$$
\end{lemma}

Let $F$ be a graph with Tur\'{a}n density
$\pi(F)=1-\frac{1}{r},$ where $\pi(F)$ is the Tur\'{a}n density of a graph $F$ and is defined as $\pi(F)=\lim\limits_{n\rightarrow \infty} \frac{ex(n,F)}{\binom{n}{2}}.$ Denote by $\mathcal{G}_{n}$ the collection of all $n$-vertex $F$-free graphs with
minimum degree more than $(\pi(F)-\epsilon)n$ and $q(\mathcal{G}_{n})=\max\{q(G):G\in \mathcal{G}_{n}\}.$
%and $\chi(F)=r+1.$
%Notice that $\pi(F)=1-\frac{1}{r}.$
%Let $\mathcal{G}_{n}$ be the collection of all $n$-vertex $F$-free graphs with
%minimum degree more than $(\pi(F)-\epsilon)n.$
If  $r\geqslant 3$ and $F$ is  a graph with  $ex(n,F)=t_{r}(n)+O(1),$  then
$\pi(F)=1-\frac{1}{r}$ as $\left(1-\frac{1}{r}\right)\frac{n^2}{2}-\frac{r}{8}\leqslant  t_{r}(n)\leqslant  \left(1-\frac{1}{r}\right)\frac{n^2}{2}.$
Recently, Zheng, Li and Li~\cite{Zheng Li Li 2025} deduced the following property of $\mathcal{G}_{n}.$
\begin{theorem}[\cite{Zheng Li Li 2025}]\label{theorem of liyongtao}
Let $r\geqslant 3$ and $F$ be a family of graphs with Tur\'{a}n density
$\pi(F)=1-\frac{1}{r}.$ For real numbers $0<\epsilon <\frac{1}{2}$
and $\sigma<\frac{\epsilon}{36}.$
%let $\mathcal{G}_{n}$ be the collection of all $n$-vertex $F$-free graphs with
%minimum degree more than $(\pi(F)-\epsilon)n.$ We denote $q(\mathcal{G}_{n})=\max\{q(G):G\in \mathcal{G}_{n}\}.$
Suppose that
there exists $N > 0$ such that for every $n \geqslant  N,$ we have
\begin{eqnarray}\label{the first inequality}
\begin{vmatrix}\operatorname{ex}(n,F)-\operatorname{ex}(n-1,F)-\pi(F)n\end{vmatrix}\leqslant \sigma n,
\end{eqnarray}
and
\begin{eqnarray}\label{the second inequality}
\begin{vmatrix}q(\mathcal{G}_n)-4\mathrm{ex}(n,F)n^{-1}\end{vmatrix}
\leqslant \sigma.
\end{eqnarray}
Then there exists $n_{0} \in \mathbb{N}$ such that for any $F$-free graph $H$ on $n \geqslant n_{0}$ vertices, we have
$$
q(H)\leqslant  q(\mathcal{G}_n).
$$
In addition, if the equality holds, then $H \in \mathcal{G}_n$.
\end{theorem}

\begin{lemma}[\cite{Cioaba Feng Tait Zhang 2020}]\label{property of cup of sets}
Let $V_{1},\cdots, V_{p}$ be finite sets. Then
$$ |V_{1}\cap\cdots\cap V_{p}|\geqslant \sum_{i=1}^{p}|V_{i}|-(p-1)\mid\bigcup_{i=1}^{p}V_{i}\mid.$$
\end{lemma}

\section{The structure of graphs  in $\mathcal{G}_{n}$}

Let $r\geqslant 2,$ $F$ be a graph %a maximal $F$-free graph is a $F$-free graph adding any new edges will result a graph containing a subgraph isomorphic to $F$.
and $ex(n,F)=t_{r}(n)+O(1).$  Then $\chi(F)=r+1$ and there exists an integer $c_{0}$ such that %$c_{0}=  \adjustlimits\limsup_{k\rightarrow\infty n\geqslant k}\{{ex(n,F)-t_r(n)}\}$.
$c_{0}=  \lim\limits_{k\rightarrow\infty}\sup\limits_{n\geqslant k}\{{ex(n,F)-t_r(n)}\}$. In fact,
since $\sup\limits_{n\geqslant k}\{ex(n,F)-t_r(n)\}$ is a decreasing integer sequence with respect to $k$ and has lower bounds,   there is an integer $n_0$ such that for any $k\geqslant n_0$, $\sup\limits_{n\geqslant k}\{{ex(n,F)-t_r(n)}\}=c_{0}$.
Notice that $c_{0}$ and  $n_{0}$ are both dependent on $F,$
let $k_0(F)=\max\{n_0,c_0+1\}$.%if $H$ is a $F$-free graphs with maximum number of edges, then   $e(H)\leqslant t_{r}(n)+c $.

\begin{theorem}\label{jin-main lemma}
    Let $1\gg\epsilon>0,$ $r\geqslant 2$ and  $F$ be a graph with  $ex(n,F)=t_{r}(n)+O(1).$  If $G\in \mathcal{G}_{n} $
   % is a  $F$-free graph on $n$ vertices with $\delta(G)\geqslant (\pi(F)-\epsilon)n$
    and $n$ is sufficiently large, then  \\
    %$(a)$ $G=G_0\cup G_{in}$, where  $G_0\subseteq K_{n_1,\cdots, n_r}$,  $G_{in}\subseteq K^c_{n_1,\cdots, n_r}$ and $\left(\frac{1}{r}-3\sqrt{\epsilon}\right)n\leqslant n_i\le\left(\frac{1}{r}+3\sqrt{\epsilon}\right)n$ for $1\leqslant i\leqslant r$.\\
    $(i)$   there exists a unique partition $V_1, \cdots, V_r$ of  $V(G)$ satisfying  that $e(V_i)\leqslant c_0$  for $1\leqslant i\leqslant r.$ Moreover, we have $\left(\frac{1}{r}-3\sqrt{\epsilon}\right)n< |V_{i}|<\left(\frac{1}{r}+3\sqrt{\epsilon}\right)n$ for $1\leqslant i\leqslant r;$  \\
    %$(b)$   there exists a unique partition $V_1, \cdots, V_r$ of  $V(G)$ satisfying  that $e(V_i)\leqslant c_0$ %and the maximum degree of $G[V_i]$ is at most $k_0-1$
    %and   where $n_{i}=|V_{i}|$ for $1\leqslant i\leqslant r.$\\
    $(ii)$ $e(G_{in})-e(G_{out})\leqslant c_0,$ where $G_{in},$ $G_{out}$ are two graphs  with the same vertex set $V(G)$ and edge sets $\cup_{i=1}^{r}E(V_{i}),$ $\cup_{i\neq j}\overline{E(V_{i},V_{j})}$, respectively. % induced by the edges of $E(K_{n_1,\cdots, n_r}-G_0)$.
\end{theorem}
\begin{proof}
Apply the graph regularity lemma to obtain an $\eta$-regular partition $V(G)=W_{1} \cup \cdots \cup W_{m}$ for some sufficiently small $\eta>0$ only depending on $\epsilon$ and $F$, to be decided later. Then the number $m$ of parts is also bounded for fixed $F$ and $\epsilon$.

Remove all edges between $ W_{i} $ and $ W_{j}$ if

(a) $\left(W_{i}, W_{j}\right)$ is not $\eta$-regular, or

(b) $d\left(W_{i}, W_{j}\right)<\epsilon / 16$, or

(c) $\min \left\{\left|W_{i}\right|,\left|W_{j}\right|\right\}<\epsilon n /(16 m)$.

Then the number of edges in (a) is no more than $\eta n^{2} \leqslant  \epsilon n^{2} / 16$, the number of edges in (b) is less than $\epsilon n^{2} / 16$, and the number of edges in (c) is less than $m \epsilon n^{2} /(16 m) \leqslant  \epsilon n^{2} / 16$. Thus, the total number of edges removed is no more than $(3 / 16) \epsilon n^{2}$. After removing all these edges, the resulting graph $G^{\prime}$ still has greater than $\pi(F)\frac{n^{2}}{2}-\frac{11\epsilon}{16}n^{2}$ edges.

Next we claim that $G^{\prime}$ is $K_{r+1}$- free.
%Suppose that $G^{\prime}$ contains a copy of $K_{r+1}$.
Otherwise, assume that the $r+1$ vertices of this $K_{r+1}$ land in $W_{i_1}, \cdots, W_{i_{r+1}}$ (allowing repeated indices). Since each pair of these sets is $\eta$-regular, has edge density no less than $ \epsilon / 16$, and each has size  no less than $ \epsilon n /(16 m)$,  applying Theorem~\ref{graph counting lemma}, %the graph counting lemma (see Theorem 2.6.4 in~\cite{Zhao 2023}),
we see that as long as
$\eta<\left(\epsilon/(16|V(F)|)\right)^{|V(F)|}$
%$\eta<\left(\frac{\epsilon}{16|V(F)|}\right)^{|V(F)|}$
%$\eta\geqslant (16 |V(F)|)^{\frac{|V(F)|}{|V(F)|-1}}$
% $\eta$ is sufficiently small in terms of $\epsilon$ and $F$,
and $n$ is sufficiently large, there exists an injective embedding of $F$ into $G^{\prime},$ where the vertices of $F$ in the $j$-th color class are mapped into $W_{i_j}$ and
 $|W_{i_{j}}|\geqslant \frac{\epsilon n}{16m}\geqslant \frac{|V(F)|}{\eta}$ for $j=1,\cdots,r+1.$
%$k\geqslant |V(F)|$.
So $G$ contains $F$ as a subgraph, a contradiction. Thus $G^{\prime}$ is $K_{r+1}$-free.

By Lemma~\ref{Furedi}, $G^{\prime}$ contains a $r$-chromatic subgraph $G^{{\prime\prime}}$ such that
$e(G^{\prime\prime})>\pi(F)\frac{n^{2}}{2}-\frac{22\epsilon}{16}n^{2}$
and $G^{\prime\prime}$ can be obtained from  $G^{\prime}$ by removing at most $\frac{11\epsilon}{16}n^{2}$ edges. Then we can remove at most $\frac{14\epsilon}{16}n^{2}$ edges from $G$ to obtain $G^{{\prime\prime}}$.
Let $U_1,U_2, \cdots, U_r$ be the $r$ parts of the $r$-chromatic subgraph $G^{{\prime\prime}}$, then $\left(\frac{1}{r}+2\sqrt{\epsilon}\right)n>|U_i|>\left(\frac{1}{r}-2\sqrt{\epsilon}\right)n$ for each $1\leqslant i\leqslant r$ as $e(G^{\prime\prime})>\pi(F)\frac{n^{2}}{2}-\frac{22\epsilon}{16}n^{2}.$
Let $k=\max\{k_{0}(F),r\}.$
\begin{claim}\label{Claim1}
If there is a vertex $u\in U_{i_{0}}$  such that $d_{U_{i_{0}}}(u)>{\epsilon}^{\frac{1}{4}} n$, then there exists just one
 $U_{j_{0}}$ such that $d_{U_{j_{0}}}(u)<k^2\epsilon^\frac{1}{4} n$, where $j_{0}\neq i_{0}$.
\end{claim}

Without loss of  generality, we can assume that $i_0=1$. Suppose for any $j\neq 1$, we have $d_{U_j}(u)\geqslant  k^2\epsilon^\frac{1}{4} n$. Let
$${X}_0=\cup _{t=1}^{r}\{\ v\ |\ \mbox{ if } v
\in U_t \mbox
{ and }
d_{U_t}(v)>{\epsilon} ^{\frac{1}{4}}n\}. $$
Since we can remove at most $\frac{14\epsilon}{16}n^{2}$ edges from $G$ to obtain $G^{\prime\prime}$,  $\sum_{i=1}^{r}e(U_{i})\leqslant  \frac{14\epsilon}{16}n^{2},$ which implies that
$$| X
_0|<2{\epsilon}^{\frac{3}{4}}n<{\epsilon} ^{\frac{1}{4}}n.$$
By $u\in U_{1}$ and $d_{U_{1}}(u)>\epsilon^{\frac{1}{4}}n,$  there are $k$ vertices in $ U_1\setminus X_0$, denoted by $u_{1t},1\leqslant t\leqslant k$, adjacent to $u$. We will do the following process to get $u_{st}\in U_s,~2\leqslant s\leqslant r,~1\leqslant t\leqslant k$. \\

\vspace{-0.2cm}
 \hspace{0.35cm}$s:=1;$

\vspace{0.05cm} \hspace{0.3cm}{\bf While} $s<r$ {\bf do begin}

\vspace{0.05cm}\hspace{0.7cm}
$X_s=\{ v\in \cup_{i=s+1}^rU_i| ~\mbox{there exists some } 1\leqslant t\leqslant k \mbox{ such that } v\mbox{ is not adjacent to }u_{st}\}$; \\
\vspace{0.05cm}\hspace{1.2cm}
Select $k$ vertices
from $ U_{s+1}\setminus (\cup_{i=0}^{s}X_i)$ adjacent to $u$, denoted by $u_{(s+1)t},~1\leqslant t\leqslant k;$\\
\vspace{0.05cm}\hspace{1.25cm} $s:=s+1$;

\vspace{0.1cm}\hspace{0.3cm}{\bf End.}
\vspace{0.25cm}

By $\delta(G)> (\pi(F)-\epsilon)n$ and each vertex of $  U_{s}\setminus (\cup_{i=0}^{s-1}X_i)$ has no more than $ {\epsilon} ^{\frac{1}{4}}n$ neighbors in $U_{s}$, so each vertex of  $  U_{s}\setminus (\cup_{i=0}^{s-1}X_i)$ has more than $(\pi(F)-\epsilon-{\epsilon} ^{\frac{1}{4}})n$ neighbors not in $U_{s}$. Since $|\cup_{\substack{i=1\\i\neq s}}^rU_i|=n-|U_s|< (\pi(F)+2\sqrt{\epsilon})n$, we have
$$|X_s|\leqslant k\left(|\cup_{\substack{i=1\\i\neq s}}^rU_i|-(\pi(F)-\epsilon-{\epsilon} ^{\frac{1}{4}})n\right)< k\left(2\sqrt{\epsilon}+\epsilon+\epsilon^{\frac{1}{4}}\right)n.$$
Since for any $j\neq 1$,  $d_{U_j}(u)\geqslant k^2\epsilon^\frac{1}{4} n$ and $$\left| \cup_{i=0}^{r-1}X_i\right|< (r-1)k\left(2\sqrt{\epsilon}+\epsilon+\epsilon^{\frac{1}{4}}\right)n+2{\epsilon}^{\frac{3}{4}}n<k^2\epsilon^{\frac{1}{4}} n,$$
we can choose vertices $u_{st},~1\leqslant t\leqslant k$ which are adjacent to $u$. This  ensures the procedure can be done until $s=r$ as $|U_t|>\left(\frac{1}{r}-2\sqrt{\epsilon}\right)n$ for $1\leqslant  t\leqslant  r$.

 By the above procedure, we have $u_{ij}$ is adjacent to $u_{st}$ for any $i\neq s$. We denote by $H$ the graph induced by $u$ and $u_{st},~1\leqslant s\leqslant r,~1\leqslant t\leqslant k.$ Then $e(H)\geqslant t_r(rk+1)+k>t_r(rk+1)+c_0\geqslant  ex(rk+1,F)$ as $rk+1>k_0(F)$, a contradiction as $H$ is $F$-free. Then there exists  one
 $U_j$ such that $d_{U_j}(u)<k^2\epsilon^\frac{1}{4} n$, where $j\neq 1$. If there are two $U_{j_1},U_{j_2},~j_1, j_2\neq 1$ such that $d_{U_{j_1}}(u),~d_{U_{j_2}}(u)<k^2\epsilon^\frac{1}{4} n$, then $$d_G(u)\leqslant n-|U_{j_1}|-|U_{j_2}|+2k^2\epsilon^\frac{1}{4}n< n-\left(\frac{2}{r}-4\sqrt{\epsilon}\right)n+2k^2\epsilon^\frac{1}{4}n<(\pi(F)-\epsilon)n< \delta(G),$$
a contradiction.

\vspace{0.25cm}

\begin{claim}\label{Claim2}
    There exists a partition $V_1,V_2, \cdots, V_r$ of  $V(G)$ such that for each $1\leqslant i\leqslant r$, $\left(\frac{1}{r}-3\sqrt{\epsilon}\right)n<|V_i|<\left(\frac{1}{r}+3\sqrt{\epsilon}\right)n$ and $e(V_i)\leqslant c_{0}.$ %and the maximum degree of $G[V_i]$ is at most $k-1$.
\end{claim}

By Claim~\ref{Claim1}, we can move all  possible  vertices  $u\in U_i$ to the corresponding vertices set  $U_j$ if $d_{U_i}(u)>\epsilon^{\frac{1}{4}}n$, $d_{U_j}(u)<k^2\epsilon^{\frac{1}{4}}n$ and $i\neq j,$ the resulting partition of $V(G)$ is $V_1,V_2, \cdots,V_r$. Then   $\left(\frac{1}{r}+3\sqrt{\epsilon}\right)n>|V_i|>\left(\frac{1}{r}-3\sqrt{\epsilon}\right)n$ for each $1\leqslant i\leqslant r$ as $|X
_0|<2{\epsilon}^{\frac{3}{4}}n<{\epsilon} ^{\frac{1}{2}}n$, and for any vertex $w\in V_i,$ we have
\begin{eqnarray}\label{procedure}
d_{V_i}(w)<k^2{\epsilon} ^{\frac{1}{4}}n +|X_0|<(k^2+1){\epsilon} ^{\frac{1}{4}}n.
\end{eqnarray}
%Firstly, we prove that the maximum degree of $G[V_i]$ is at most $k-1$. Suppose not, without loss of the generality, we can assume that the maximum degree of $G[V_1]$ is at least $k$. Let $v\in V_1$, $d_{V_1}(v)\geqslant k$ and $v_{1t}\in V_1,~1\leqslant t\leqslant k $ are adjacent to $v$. We can do the same procedure as in Claim~\ref{Claim1} to get a graph of order $rk+1$, which contains more than $t_r(rk+1)+k$ edges, a contradiction.

Next we prove that $e(V_i)\leqslant c_{0}$  for each $1\leqslant i\leqslant r$. Suppose not, without loss of  generality, we can assume  $e(V_1)\geqslant c_{0}+1$. Then  choose $2(k+1)$ vertices $w_{1t}\in V_1,~1\leqslant t \leqslant 2(k+1)$ such that the graph induced by these vertices contains at least $c_{0}+1$ edges. By~(\ref{procedure}), we  can also excute the procedure of Claim~\ref{Claim1} to get an $F$-free graph of order $2r(k+1)$, which contains at least $t_r(2r(k+1))+c_{0}+1$ edges, a contradiction.

\begin{claim}
    Let $G_{in},$ $G_{out}$ be two graphs  with the same vertex set $V(G)$ and edge sets $\cup_{i=1}^{r}E(V_{i}),$ $\cup_{i\neq j}\overline{E(V_{i},V_{j})}$, respectively. %$G_{in}=\cup_{i=1}^rG[V_i]$ and $G_{out}$ be the graph induced by the edges missed between $V_i$ and $V_j$, where $i\neq j$.
    Then $e(G_{in})-e(G_{out})\leqslant c_{0}.$
\end{claim}

By Claim~\ref{Claim2}, $e(G_{in})= \sum_{i=1}^re(V_i)\leqslant rc_{0}.$ If $e(G_{in})-e(G_{out})\geqslant c_{0}+1,$ then $ e(G_{out})\leqslant (r-1)c_{0}-1.$ Denote by $U$ the set of the vertices which are incident with at least one edge of $E(G_{in})$ or $E(G_{out})$.  We have $|U|\leqslant 4rc_{0}-2c_{0}-2$. Thus, we can choose a vertex subset $W$ of $V(G)$ containing $rk( 4rc_{0}-2c_{0}-2)$ vertices  such that $U\subseteq W$ and $|W\cap V_i|=|W\cap V_j|,~i\neq j$. Then  $e(W)=t_r(rk(4rc_{0}-2c_{0}-2))+e(G_{in})-e(G_{out})\geqslant t_r(rk(4rc_{0}-2c_{0}+2))+c_{0}+1$, a contradiction.

\begin{claim}
The partition $V_1, \cdots, V_r$ of  $V(G)$ in Claim~2  is the unique partition satisfying  that $e(V_i)\leqslant c_0$  for $1\leqslant i\leqslant r.$
\end{claim}
%Without loss of generality, assume that $n_{1}\geqslant \cdots \geqslant n_{r}.$
Suppose $V_1',V_2', \cdots, V_r'$ is  a different partition of $V(G)$
 satisfying  that $e(V_i')\leqslant c_0$  for $1\leqslant i\leqslant r.$
 %and $|V_1'|\geqslant \cdots \geqslant |V_r'|.$
 Combining with $e(G)\geqslant  \frac{1}{2}(\pi(F)-\epsilon)n^2,$ we have $\left(\frac{1}{r}-3\sqrt{\epsilon}\right)n<|V_{i}'|<\left(\frac{1}{r}+3\sqrt{\epsilon}\right)n.$ Then there exists some  $V_{i}'$ contains two vertices $v$ and $w$ satisfying that $v\in V_{j_{1}}$ and $w\in V_{j_{2}},$ where $j_{1}\neq j_{2}.$ %Assume that $v\in V_{j}'.$
%Let $i_{0}$ be the smallest index
Notice that $|N(v)\cap V_{j_{1}}|\leqslant  c_{0}$  and $|N(v)\cap V_{i}'|\leqslant  c_{0},$ which implies that
\begin{eqnarray*}
|V_{i}'\cup V_{j_{1}}|&=&|(V_{i}'\cup V_{j_{1}})\cap(N(v)\cup N^{c}(v))|=|((V_{i}'\cup V_{j_{1}})\cap N(v))\cup ((V_{i}'\cup V_{j_{1}})\cap N^{c}(v))| \\
&\leqslant  & |(V_{i}'\cup V_{j_{1}})\cap N(v)|+|N^{c}(v)| \leqslant  2c_{0}+|N^{c}(v)|.
\end{eqnarray*}
Thus,
%$|N^{c}(v)\cap V_{j_{1}}|\geqslant  |V_{j_{1}}|- c_{0}$ and $|N^{c}(v)\cap V_{i}'|\geqslant c_{0}.$ Thus,
\begin{eqnarray*}
|V_{i}'\cap V_{j_{1}}|&=&|V_{i}'|+|V_{j_{1}}|-|V_{i}'\cup V_{j_{1}}|\geqslant |V_{i}'|+|V_{j_{1}}|-(2c_{0}+|N^{c}(v)|)\\
&\geqslant & |V_{i}'|+|V_{j_{1}}|-(2c_{0}+n-\delta(G))>
\frac{n}{r}-(6\sqrt{\epsilon}+\epsilon)n-2c_{0}.
\end{eqnarray*}
Also,
\begin{eqnarray*}
|N(w)\cap V_{j_{1}}|&=&|N(w)|+|V_{j_{1}}|-|N(w)\cup V_{j_{1}}|=|N(w)|+|V_{j_{1}}|-|(N(w)\cap V_{j_{2}})\cup (N(w)\cap V_{j_{2}}^{c})\cup V_{j_{1}}|\\
&\geqslant& |N(w)|+|V_{j_{1}}|-(|N(w)\cap V_{j_{2}}|+ |(N(w)\cap V_{j_{2}}^{c})\cup V_{j_{1}}|)\\
&\geqslant& |N(w)|+|V_{j_{1}}|-(c_{0}+ n-|V_{j_{2}}|)
>\frac{n}{r}-(6\sqrt{\epsilon}+\epsilon)n-c_{0}.
\end{eqnarray*}
Thus,
\begin{eqnarray*}
|N(w)\cap V_{i}'|&\geqslant& |N(w)\cap V_{i}'\cap V_{j_{1}}|=|(N(w)\cap V_{j_{1}})\cap (V_{i}'\cap V_{j_{1}})|\\
&=& |N(w)\cap V_{j_{1}}|+ |V_{i}'\cap V_{j_{1}}|-|(N(w)\cap V_{j_{1}})\cup (V_{i}'\cap V_{j_{1}})|\\
&=& |N(w)\cap V_{j_{1}}|+ |V_{i}'\cap V_{j_{1}}|-|(N(w)\cup V_{i}')\cap V_{j_{1}}| \\
&\geqslant& |N(w)\cap V_{j_{1}}|+ |V_{i}'\cap V_{j_{1}}|-|V_{j_{1}}| >
\frac{n}{r}- 15\sqrt{\epsilon}n-2\epsilon n-3c_{0}>c_{0},
\end{eqnarray*}
where the last inequality holds as $n$ is sufficiently large. A contradiction to $e(V_{i}')\leqslant  c_{0}.$ We complete the proof.
\end{proof}
\begin{remark}\label{lemma 1.3}
Let $G\in \mathcal{G}_{n}$ be a  graph. By Theorem~\ref{jin-main lemma}, there exists
%an unique complete multipartite graph $K_{n_{1},\cdots,n_{r}}$ and
a unique partition $V_{1},\cdots,V_{r}$ of $V(G)$  satisfying that   $G=G_{0}\cup G_{in},$  $e(G_{in})-e(G_{out})\leqslant c_{0}$ and $e(V_{i})\leqslant  c_{0}$ for $1\leqslant i\leqslant r,$  where  $G_0\subseteq K_{n_1,\cdots, n_r}$,  $G_{in}\subseteq K^c_{n_1,\cdots, n_r},$  $\left(\frac{1}{r}-3\sqrt{\epsilon}\right)n< n_i=|V_{i}|<\left(\frac{1}{r}+3\sqrt{\epsilon}\right)n$ %and $e(V_{i})\leqslant  c_{0}$
for $1\leqslant i\leqslant r.$ Thus, if $G\in \mathcal{G}_{n},$ then the  three graphs $G_{0},$ $G_{in}$ and $G_{out}$ associated with $G$ are defined,    %$e(G_{in})-e(G_{out})\leqslant c_{0},$  $G$ can be written as  $G_{0}\cup G_{in}$ and
where
$G_{0},$ $G_{in}$ and $G_{out}$ are three graphs corresponding to the partition  $V_{1},\cdots,V_{r}$ of $V(G)$ and  satisfying  the above conditions.
\end{remark}

\begin{corollary}\label{corollary}
Let  $F$ be  a graph
with  $ex(n,F)=t_r(n)+c_0$.
If $G$ is an $F$-free graph of order $n$ with $e(G)=ex(n,F)$ and $n$ is sufficiently large, then $G=G_0\cup G_{in},$  $ex(n-1,F)=t_r(n-1)+c_0,$  $e(G_{in})-e(G_{out})= c_0$ and  $d_{G_{out}}(v)\leqslant  c_0+1$ for any vertex $v\in V(G),$ where $G_0\subseteq T_{r}(n)$, $G_{in}\subseteq T_{r}^c(n).$
\end{corollary}
\begin{proof}
If $\delta(G)<\lfloor(1-\frac{1}{r})n\rfloor$, then select $w\in V(G)$ satisfying that $d(w)=\delta(G)$. Let $G_{n-1}=G-w.$  Since $G_{n-1}\subseteq G$ and $G$ is $F$-free, $G_{n-1}$ is $F$-free. It is easy to see that
$e(G_{n-1})=e(G)-\delta(G)\geqslant t_r(n)+c_0-\lfloor(1-\frac{1}{r})n\rfloor+1>t_r(n-1)+c_0,$
which is a contradiction as $n$ is sufficiently large and $ex(n-1,F)\leqslant  t_r(n-1)+c_0$. So $\delta(G)\geqslant\lfloor(1-\frac{1}{r})n\rfloor$.

Moreover, since $\delta(G)\geqslant \lfloor(1-\frac{1}{r})n\rfloor,$
by Theorem~\ref{jin-main lemma},
$G=G_0\cup G_{in}$, $G_0\subseteq K_{n_1,\cdots,n_r}$ and $G_{in}\subseteq K_{n_1,\cdots,n_r}^c$.
%\textcolor{red}{
Since $n$ is sufficiently large and $e(V_i)\leqslant  c_0$, there is a vertex $w_i\in V_{i}$ for each $1\leqslant  i\leqslant  r$ such that $d_{V_i}(w_i)=0$.
If there exist two integers $i$ and $j$ satisfying that $|n_i-n_j|> 1,$ where $1\leqslant  i,j\leqslant  r$, then $\delta(G)\leqslant \min_{1\leqslant  i\leqslant  r}\{n-n_i\}<\lfloor(1-\frac{1}{r})n\rfloor$, a contradiction to $\delta(G)\geqslant\lfloor(1-\frac{1}{r})n\rfloor$.
Thus  $|n_i-n_j|\leqslant  1 $ for any $1\leqslant  i,j\leqslant  r$ and  there is a vertex $w'\in V(G)$ such that $d(w')=\delta(G)= \lfloor(1-\frac{1}{r})n\rfloor,$ which deduces that $K_{n_1,\cdots,n_r}=T_{r}(n)$.
Obviously,
for any vertex $w\in V(G),$ $d_{G_{out}}(w)\leqslant  c_0+1$ as $\delta(G)\geqslant \lfloor(1-\frac{1}{r})n\rfloor$ and $e(V_i)\leqslant c_0$ for each $1\leqslant i\leqslant r.$

Now
let $G_{n-1}'=G-w'.$  It is easy to see that $G_{n-1}'$ is $F$-free and $e(G_{n-1}')=e(G)-\delta(G)= t_r(n)+c_{0}-\lfloor(1-\frac{1}{r})n\rfloor =t_r(n-1)+c_0.$ Also, by
$e(G_{n-1}')\leqslant ex(n-1,F)\leqslant t_r(n-1)+c_0$, we have $ex(n-1,F)=t_r(n-1)+c_0.$
Since $G=G_0\cup G_{in}$ and $e(G)=t_{r}(n)+c_0,$   then $t_{r}(n)+c_0=e(G)=t_{r}(n)+e(G_{in})-e(G_{out}),$ which implies that $e(G_{in})-e(G_{out})=c_0.$
\end{proof}

\begin{remark}\label{c0}
By the definition of $c_{0}$ and Corollary~\ref{corollary}, if $n$ is sufficiently large, then $ex(n,F)=t_{r}(n)+c_{0}.$
\end{remark}

Recall that a graph $G$ is $F$-free if $G$ does not contain $F$ as a  subgraph,  we say a graph  $G$ is {\it $F$-saturated} if $G$ is $F$-free but $G+e$ does contain  a copy of $F$ for each edge $e\in E(\overline{G}).$ Obviously, for an $F$-saturated graph $G$ of order $n,$ we have $|E(G)|\leqslant  ex(n,F).$

%A {\it maximal $F$-free graph} is an $F$-free graph adding any new edges  results in a graph containing a subgraph isomorphic to $F$.
\begin{corollary}\label{property of edge maximal}
For a graph $F$
with  $ex(n,F)=t_r(n)+O(1)$,
let $G=G_0\cup G_{in}\in \mathcal{G}_n$ be an $F$-saturated %maximal $F$-free
graph of order $n$ %with the maximum number of
%, where $G_0\subseteq K_{n_1,\cdots,n_r},G_1\subseteq K_{n_1,\cdots,n_r}^c $ are obtained from Theorem~\ref{jin-main lemma}
and
the corresponding partition of  $V(G)$ is $V_1,V_2,\cdots,V_r$. Let $A_i=\{v\in V_i \ | \ d_{V_i}(v)\geqslant 1\}$ and $B_i=V_i\setminus A_i$ for any $1\leqslant i\leqslant r$. If $n$ is sufficiently large, then  $|A_i|\leqslant 2c_0,$ $e(G_{out})\leqslant  2(r-1)rc_{0}^2$ and  any vertex $u\in B_i$ is adjacent to every vertex of $V(G)\setminus V_i$ for $1\leqslant i\leqslant r.$
\end{corollary}
\begin{proof}
By Theorem~\ref{jin-main lemma}, $e(V_i)\leqslant c_0$, which implies that $|A_i|\leqslant 2c_0$. Suppose that there are two vertices $u\in B_i$ and $v\in V(G)\setminus V_i$ satisfying  that $uv\notin E(G)$. Since $G$ is an $F$-saturated graph,  the graph $G'=G+\{uv\}$ contains a copy of $F$, denoted by $F'$. Since $G$ is $F$-free, $F'$ contains the edge $uv$. Let $N_{F'}(u)=\{w_1,w_2,\cdots, w_s\}$. Since $u\in B_i$, we have  $w_1,w_2,\cdots,w_s$ are not in $V_i$. We can assume that $w_j\in V_{i_j},~1\leqslant j\leqslant s.$

For any vertex $w_j$,  we have
$$|N_G(w_j)\cap V_i|=|N_G(w_j)|+|V_i|-|N_G(w_j)\cup V_i|\geqslant \delta(G)+|V_i|-n+|V_{i_j}|-c_0>\left(\frac{1}{r}-\epsilon-6\sqrt{\epsilon}\right)n-c_0,$$
which follows from $(\frac{1}{r}-3\sqrt{\epsilon})n< |V_i|,|V_{i_j}|< (\frac{1}{r}+3\sqrt{\epsilon})n$, $N_G(w_j)\cup V_i\subseteq (V\setminus V_{i_j})\cup (N _G(w_j)\cap V_{i_j} )$  and $|N _G(w_j)\cap V_{i_j}| \leqslant e(G[V_{i_j}]) \leqslant c_0$.
By Lemma~\ref{property of cup of sets},
\begin{eqnarray*}
|\cap _{j=1}^s(N_G(w_j)\cap V_i)|&\geqslant & \sum_{j=1}^s|N_G(w_j)\cap V_i|-(s-1)|\cup_{j=1}^s(N_G(w_j)\cap V_i)|\\
&\geqslant &\sum_{j=1}^s|N_G(w_j)\cap V_i|-(s-1)| V_i|\\
&\geqslant &s\left(\frac{1}{r}-\epsilon-6\sqrt{\epsilon}\right)n-sc_0-(s-1)\left(\frac{1}{r}+3\sqrt{\epsilon}\right)n\\
&=&\left(\frac{1}{r}-s\epsilon-(9s-3)\sqrt{\epsilon}\right)n-sc_0\geqslant|V(F)|+2+|A_{i}|,
\end{eqnarray*}
the last inequality holds as $n$ is sufficiently large and $\epsilon$ is sufficiently small. We can find a vertex $u'\in B_i, u'\neq u$, which is adjacent to all vertices $w_j,1\leqslant j\leqslant s$. Replace $u$ with $u'$ in $F'$, we can find a copy of $F$ in $G$, a contradiction. Thus any vertex $u\in B_i$ is adjacent to every vertex of $V(G)\setminus V_i$ for $1\leqslant i\leqslant r$.

Furthermore, by $|A_i|\leqslant 2c_0$ and  any vertex $u\in B_i$ is adjacent with every vertex of $V(G)\setminus V_i$ for $1\leqslant i\leqslant r,$ the edges of $G_{out}$ are only occurred between $A_{i}$ and $A_{j},$ which implies that
$e(G_{out})\leqslant  2(r-1)rc_{0}^2.$
\end{proof}
%\begin{lemma}\label{the upper bound of q}
%Let $1\gg \epsilon >\eta >0,$  $F$ be a graph with $\chi(F)=r+1$  and $G$ be a graph on $n$ vertices with minimum degree $\delta(G)\geqslant (\pi(F)-\epsilon)n,$ there is a partition of $V(G),$ denoted by $\{V_{1},\cdots,V_{r}\},$
%and a subset $W$ of $V(G)$ such that $|W|<\eta n,$  $|V_{i}|>\frac{n}{r}-\epsilon^{1/4}n$ and $d_{\overline{V_{i}}}(w)\geqslant (1-\frac{1}{r}-\epsilon^{\frac{1}{8}})n$ for $1\leqslant  i \leqslant  r.$ If there is vertex $u\in V_{1}$ containing more than $|V(F)|$ neighbors in $V_{1}\cap \overline{W},$ then $G$ contains a copy of $F$ for sufficiently large $n.$
%\end{lemma}
%\begin{proof}
%\end{proof}
\section{Proof of Theorem~\ref{main result}}
In this section, we give the proof of Theorem~\ref{main result}.
The  spectral radius of a nonnegative matrix $A,$ denoted
by $\rho(A)$, is the largest eigenvalue of $A.$
By Perron-Frobenius theorem, there exists a nonnegative vector $\mathbf{x}$ of $Q(G)$ corresponding to $q(G).$
We call it    the {\it Perron vector} if   it satisfies  $\max_{i\in V(G)}\{x_{i}\}=1.$
For a complete $r$-partite graph $K_{n_{1},\cdots,n_{r}}$ on $n$ vertices,
if $r=2,$
then
%Notice that for all complete bipartite graphs
%$K_{n_{1},n-n_{1}}$ , we have
$q(K_{n_{1},n_{2}})=n$ for any two positive integers $n_{1}$ and $n_{2}$ satisfying that $n_{1}+n_{2}=n,$  which is totally different from the case for $r\geqslant 3.$ We %can %
shall
see it in the following lemma.
\begin{lemma}\label{difference of two type of partitions}
Let $r\geqslant 3$ and $n_{1}\geqslant \cdots\geqslant  n_{r}\geqslant  1$ be integers satisfying that $\sum_{i=1}^{r}n_{i}=n.$ If there exist two integers $1\leqslant  i<j\leqslant  r$ satisfying that $n_{i}-n_{j}\geqslant 2,$ then
$(i)$ $q(K_{n_{1},\cdots,n_{i}-1,\cdots,n_{j}+1,\cdots,n_{r}})>q(K_{n_{1},\cdots,n_{i},\cdots,n_{j},\cdots,n_{r}});$
$(ii)$ $q(T_{r}(n))>q(K_{n_{1},\cdots,n_{r}})+\frac{2(r-2)}{r^2n}.$
\end{lemma}
\begin{proof}
(i)
Let $K=K_{n_{1},\cdots,n_{r}}$ and $K'=K_{n_{1},\cdots,n_{i}-1,\cdots,n_{j}+1,\cdots,n_{r}}$.  Let $\pi=\{V_{1},\cdots,V_{r}\}$ (resp. $\pi'=\{V_{1}',\cdots,V_{r}'\}$) be the partition consisting of all partite sets of  $K$ (resp. $K'$), where $|V_{\ell}|=n_{\ell}$ for $1\leqslant  \ell \leqslant  r,$ $|V_{\ell}'|=n_{\ell}$ for $1\leqslant  \ell \leqslant  r,$ $\ell\neq i,j,$ $|V_{i}'|=n_{i}-1$ and $|V_{j}'|=n_{j}+1.$
The quotient matrix of $Q(K)$ corresponding to the partition  $\pi$ is
$$Q(K)/\pi=\begin{bmatrix}
     n-n_{1}&\cdots & n_{i} & \cdots  &  n_{j} & \cdots & n_{r}\\
      \vdots&\cdots & \vdots & \cdots  &  \vdots& \cdots  &  \vdots\\
     n_{1}&\cdots & n-n_{i} & \cdots  &  n_{j} & \cdots &  n_{r}\\
     \vdots&\cdots & \vdots & \cdots  &  \vdots& \cdots  &  \vdots\\
     n_{1}&\cdots & n_{i} & \cdots  &  n-n_{j} & \cdots&  n_{r}\\
      \vdots&\cdots & \vdots & \cdots  &  \vdots& \cdots  &  \vdots\\
     n_{1}&\cdots & n_{i} & \cdots  &  n_{j} & \cdots&  n-n_{r}
\end{bmatrix}.$$
And the quotient matrix of $Q(K')$ corresponding to the partition $\pi'$ can be deduced similarly.
By Theorem~8.2.8 in \cite{Horn Johnson 2013}, there exist two positive vectors $\mathbf{x}$ and $\mathbf{y}$ such that $(Q(K)/\pi)\mathbf{x}=\rho(Q(K)/\pi)\mathbf{x}$ and $\mathbf{y}^{T}(Q(K')/\pi')=\rho(Q(K')/\pi')\mathbf{y}^{T},$ where $\sum_{i=1}^{r}x_{i}=1$ and $\sum_{i=1}^{r}y_{i}=1.$

It follows from the irreducibility of $Q(K),$ $Q(K'),$ Lemma~2.3.1 in \cite{Brouwer2012} and Theorem~8.3.4 in \cite{Horn Johnson 2013} that $q(K)=\rho(Q(K)/\pi)$ and $q(K')=\rho(Q(K')/\pi').$
Since $(Q(K)/\pi)\mathbf{x}=q(K)\mathbf{x},$ we have
\begin{equation}
\left\{
\begin{array}{ll}
q(K)x_{i} = (n-n_{i})x_{i}+\sum\limits_{\substack{\ell=1 \\ \ell \neq i,j}}^{r}n_{\ell}x_{\ell}+n_{j}x_{j} \nonumber\\
q(K)x_{j} = (n-n_{j})x_{j}+\sum\limits_{\substack{\ell=1 \\ \ell \neq i,j}}^{r}n_{\ell}x_{\ell}+n_{i}x_{i},
\end{array}
\right.
\end{equation}
which implies that
\begin{eqnarray}\label{the eigencomponent of x}
x_{i}=\frac{q(K)-n+2n_{j}}{q(K)-n+2n_{i}}x_{j}.
\end{eqnarray}
%Obviously, $x_{i}< x_{j}$ since $n_{i}- n_{j}\geqslant2.$ %and $x_{j}>\frac{1}{r}.$

Moreover, by  $\mathbf{y}^{T}Q(K')/\pi'=q(K')\mathbf{y}^{T}$ and $\sum_{i=1}^{r}y_{i}=1,$ we have
\begin{equation}
\left\{
\begin{array}{ll}
q(K')y_{i} = (n_{i}-1)(1-y_{i})+(n-(n_{i}-1))y_{i} \nonumber\\
q(K')y_{j} = (n_{j}+1)(1-y_{j})+(n-(n_{j}+1))y_{j},
\end{array}
\right.
\end{equation}
which implies that
\begin{eqnarray*}
y_{i}=\frac{n_{i}-1}{q(K')-n+2(n_{i}-1)}, \quad y_{j}=\frac{n_{j}+1}{q(K')-n+2(n_{j}+1)}.
\end{eqnarray*}
Since $r\geqslant 3,$ we have $K'$ contains $K_{n_{1},n-n_{1}}$ as its proper subgraph, which implies that $q(K')>n$
and $y_{j}<\frac{1}{2}$.
Thus,
\begin{eqnarray}\label{the relation of y_{i}}
y_{i}&=&\frac{(n_{i}-1)(q(K')-n+2(n_{j}+1))}{(n_{j}+1)(q(K')-n+2(n_{i}-1))}y_{j}%\nonumber\\
%&=&y_{j}+\frac{(n_{i}-1)-(n_{j}+1)}{n_{j}+1}y_{j}-\frac{2((n_{i}-1)-(n_{j}+1))}{q(K')-n+2(n_{i}-1)}y_{j}-\frac{2((n_{i}-1)-(n_{j}+1))^2}{(n_{j}+1)(q(K')-n+2(n_{i}-1))}y_{j},\nonumber\\
=y_{j}+\frac{(n_{i}-n_{j}-2)(q(K')-n)}{(n_{j}+1)(q(K')-n+2(n_{i}-1))}y_{j}.
\end{eqnarray}
Obviously, $y_{i}\geqslant y_{j}$ as $n_{i}- n_{j}\geqslant  2.$% and $y_{j}<\frac{1}{r}.$

Therefore, we have
\begin{eqnarray}\label{the difference of spectral radius}
&&q(K')-q(K)>(q(K')-q(K))\mathbf{y}^{T}\mathbf{x}=(\rho(Q(K')/\pi')-\rho(Q(K)/\pi))\mathbf{y}^{T}\mathbf{x} \nonumber\\
&=& \mathbf{y}^{T}(Q(K')/\pi')\mathbf{x}-\mathbf{y}^{T}(Q(K)/\pi) \mathbf{x}=\mathbf{y}^{T}(Q(K')/\pi'-Q(K)/\pi)\mathbf{x}\nonumber\\
&=& (-y_{1}-\cdots -y_{i-1}+y_{i}-y_{i+1}-\cdots- y_{r})x_{i}+(y_{1}+\cdots+y_{j-1}-y_{j}+y_{j+1}+\cdots+y_{r})x_{j}\nonumber\\
&=& (-(1-y_{i})+y_{i})x_{i}+(1-y_{j}-y_{j})x_{j}=(2y_{i}-1)x_{i}+(1-2y_{j})x_{j}\nonumber\\
%&=& (2y_{j}-1)x_{i}+(1-2y_{j})x_{j}\nonumber\\
%&&+2(n_{i}-n_{j}-2))y_{j}x_{i}\left(\frac{1}{n_{j}+1}-\frac{2}{q(K')-n+2(n_{i}-1)}-\frac{2((n_{i}-1)-(n_{j}+1))}{(n_{j}+1)(q(K')-n+2(n_{i}-1))}\right)\nonumber\\
&=& (1-2y_{j})(x_{j}-x_{i})+2(n_{i}-n_{j}-2)y_{j}x_{i}\frac{q(K')-n}{(n_{j}+1)(q(K')-n+2(n_{i}-1))}\nonumber\\
&\geqslant & (1-2y_{j})(x_{j}-x_{i})= (1-2y_{j})\left(1- \frac{q(K)-n+2n_{j}}{q(K)-n+2n_{i}}\right)x_{j}\nonumber\\
&=& (1-2y_{j})\frac{2(n_{i}-n_{j})}{q(K)-n+2n_{i}}x_{j}>0,
\end{eqnarray}
where the last inequality holds as $y_{j}<\frac{1}{2},$ $n_{i}-n_{j}\geqslant 2$ and $q(K)>n.$
We complete the proof of (i).

(ii)
Let  $T_{r}(n)=K_{n_{1}',\cdots,n_{r}'}$ and $K''=K_{n_{1}'+1,\cdots,n_{r}'-1},$ where $\lceil \frac{n}{r}\rceil=n_{1}'\geqslant  \cdots\geqslant  n_{r}'=\lfloor \frac{n}{r}\rfloor.$
Let $\pi=\{V_{1},\cdots,V_{r}\}$ (resp. $\pi'=\{V_{1}',\cdots,V_{r}'\}$) be the  partition consisting of all partite sets  of  $K''$ (resp. $T_{r}(n)$), where $|V_{\ell}|=n_{\ell}'$ for $1\leqslant  \ell \leqslant  r,$ $\ell\neq 1,r,$ $|V_{1}|=n_{1}'+1,$ $|V_{r}|=n_{r}'-1$ and  $|V_{\ell}'|=n_{\ell}'$ for $1\leqslant  \ell \leqslant  r.$
Obviously, $n_{1}'+1\geqslant n_{2}'\geqslant \cdots \geqslant  n_{r}'-1.$
Similarly, there exist two positive vectors $\mathbf{x}$ and $\mathbf{y}$ such that $(Q(K'')/\pi)\mathbf{x}=\rho(Q(K'')/\pi)\mathbf{x}$ and $\mathbf{y}^{T}(Q(T_{r}(n))/\pi')=\rho(Q(T_{r}(n))/\pi')\mathbf{y}^{T},$ where $\sum_{i=1}^{r}x_{i}=1$ and $\sum_{i=1}^{r}y_{i}=1.$
Using the same technique as the proof in (i), by~(\ref{the eigencomponent of x}), (\ref{the relation of y_{i}})  and $n_{1}'+1\geqslant n_{2}'\geqslant \cdots \geqslant n_{r}'-1$, we have $y_{r}\leqslant \frac{1}{r}$ and $x_{r}>\frac{1}{r}.$

Then by~(\ref{the difference of spectral radius}), we have
$$q(T_{r}(n))-q(K'')>(1-2y_{r})\frac{2((n_{1}'+1)-(n_{r}'-1))}{q(K'')-n+2(n_{1}'+1)}x_{r}>\frac{2(r-2)}{r^2n},$$
where the last inequality holds as $n<q(K'')\leqslant  2(n-1),$ $y_{r}\leqslant \frac{1}{r}<x_{r}$ and $r\geqslant 3.$
Combining with (i), we complete the proof of (ii).
\end{proof}
%$\overline{K}_{\alpha n}$
\begin{lemma}\label{the upper bound of signless laplacian spectral radius}
Let $G=G_1\vee G_2$ be a graph on $n$ vertices, $|V(G_1)|=\alpha n$ and $e(G_1)\leqslant  c_1,$ where $\alpha,$  $c_1$ are two positive constants and $n>\frac{2c_{1}}{\alpha}$. Then
$$q(G)< q(\overline{K}_{\alpha n}\vee G_{2})+\frac{4c_1(1-\alpha)n}{(q(G)-(1-\alpha)n-2c_1)^{2}}\leqslant  q(\overline{K}_{\alpha n}\vee G_{2})+\frac{4c_1(1-\alpha)n}{(\alpha n-2c_1)^{2}}. $$
\end{lemma}
\begin{proof}
Let $\mathbf{x}$ be the nonnegative unit eigenvector of $G$ corresponding to
$q(G)$ and $  x_{v}=\max_{w\in V(G_{1})}\{x_{w}\}$, where $v\in V(G_1)$. By
$Q(G)\mathbf{x}=q(G)\mathbf{x}$, we have
\begin{eqnarray*}
q(G)x_{v}&=& d(v)x_{v}+\sum_{w\sim v}x_{w}
=  d(v)x_v+  \sum\limits_{\substack{w\sim v \\ w\in V(G_1)}}x_w   +  \sum_{w\in V(G_2)}x_w  \\
&\leqslant & \left[(1-\alpha)n+c_1\right]x_{v}+ c_1 x_{v}+\sum_{w\in V(G_{2})}x_{w} \\
&\leqslant & [(1-\alpha )n+2c_1]x_{v}+\sqrt{\sum_{w\in V(G_{2})}x_{w}^{2}}\cdot\sqrt{\sum_{w\in V(G_{2})}1}
 <  [(1-\alpha)n+2c_1]x_v+\sqrt{(1-\alpha)n},
\end{eqnarray*}
which implies that
$$x_v< \frac{\sqrt{(1-\alpha)n}}{q(G)-(1-\alpha)n-2c_1}\leqslant    \frac{\sqrt{(1-\alpha)n}}{\alpha n-2c_1},$$
where the last inequality holds as $ q(G)\geqslant q(K_{\alpha n, (1-\alpha)n})= n$ and $n>\frac{2c_{1}}{\alpha}$. Thus,
$$q(G)= \mathbf{x}^T Q(G)\mathbf{x}=  \mathbf{x}^TQ(\overline{K}_{\alpha n}\vee G_2)\mathbf{x}+\sum_{ij\in E(G_{1})}(x_{i}+x_{j})^2
< q(\overline{K}_{\alpha n}\vee G_{2})+\frac{4c_1(1-\alpha )n}{(\alpha n-2c_1)^{2}}.$$
Then the assertion holds.
\end{proof}

In the following lemmas,
  assume that $F$ is  a graph with $r\geqslant 3,$ $ex(n,F)=t_{r}(n)+O(1),$  let $G\in \mathcal{G}_{n}$  and $q(G)=q(\mathcal{G}_{n}).$
 By Remark~\ref{lemma 1.3}, there exists
%an unique complete multipartite graph $K_{n_{1},\cdots,n_{r}}$ and
a unique partition $V_{1},\cdots,V_{r}$ of $V(G)$  satisfying that   $G=G_{0}\cup G_{in},$  $e(G_{in})-e(G_{out})\leqslant c_{0}$ and $e(V_{i})\leqslant  c_{0}$ for $1\leqslant i\leqslant r,$  where  $G_0\subseteq K_{n_1,\cdots, n_r}$,  $G_{in}\subseteq K^c_{n_1,\cdots, n_r},$  $\left(\frac{1}{r}-3\sqrt{\epsilon}\right)n< n_i=|V_{i}|<\left(\frac{1}{r}+3\sqrt{\epsilon}\right)n$ %and $e(V_{i})\leqslant  c_{0}$
for $1\leqslant i\leqslant r.$
 By Perron-Frobenius theorem, adding any edge to a graph  will not decrease its signless Laplacian spectral radius. According to the definition of $F$-saturated graph,  $G$ is   an
$F$-saturated graph.

\begin{lemma}\label{the size of each partition}
%Let $G\in \mathcal{G}_{n}$ and $q(G)=q(\mathcal{G}_{n}).$
%Then $G=G_{0}\cup G_{1},$ where  $G_{0}$ is a subgraph of $K_{n_1,\cdots,n_{r}},$ $G_{1}$ is a subgraph of $K_{n_1,\cdots,n_{r}}^{c}.$
If $n$ is  sufficiently large,
 then there exists a constant  $c_2$ %$c_2>\left\lceil 3\sqrt{\frac{(r+1)^2}{2}+\frac{2rc_{0}(r^8-r^7+r^4)}{(r^3-2)^2}}\right\rceil$ and
satisfying that  $n+c_{2}$ is a multiple of $r$ and   $\frac{n+(r-1)^2c_2}{r}\geqslant  n_{1} \geqslant \cdots \geqslant  n_{r}>\frac{n-(r-1)c_2}{r}.$
%where $r\geqslant 3,$

\end{lemma}
\begin{proof}
%By Theorem~\ref{jin-main lemma}, $G=G_{0}\cup G_{in},$ where  $G_{0}\subseteq K_{n_1,\cdots,n_{r}},$ $G_{in}\subseteq K_{n_1,\cdots,n_{r}}^{c},$ $n_{1}\geqslant  \cdots \geqslant n_{r},$  $\left(\frac{1}{r}-3\sqrt{\epsilon}\right)n\leqslant n_i\le\left(\frac{1}{r}+3\sqrt{\epsilon}\right)n$ and $e(V_{i})\leqslant  c_{0}$ for $1\leqslant i\leqslant r.$
Assume that $n_{r}=\alpha n$ and $0<\epsilon\leqslant  \frac{1}{9r^8}.$  Obviously, $ n_{1}\geqslant  n_{2}\cdots \geqslant n_{r}=\alpha n$ and $\alpha>  \frac{1}{r}-3\sqrt{\epsilon}\geqslant  \frac{1}{r}-\frac{1}{r^4}.$ Since $e(V_{i})\leqslant  c_{0}$ for any $1\leqslant i\leqslant r,$  using Lemma~\ref{the upper bound of signless laplacian spectral radius},
 \begin{eqnarray*}
q(G)&<& q(K_{n_{1},\cdots,n_{r}})+\frac{4rc_{0}(1-\alpha) n}{(\alpha n-2c_{0})^2}
<  q(K_{n_{1},\cdots,n_{r}})+\frac{4rc_{0}(1-\frac{1}{r}+\frac{1}{r^4})n}{(\frac{1}{r}-\frac{1}{r^4}-\frac{2c_{0}}{n})^2n^2}\\
&= & q(K_{n_{1},\cdots,n_{r}})+\frac{4rc_{0}\frac{r^8-r^7+r^4}{(r^3-1-\frac{2c_{0}r^4}{n})^2}}{n}.
\end{eqnarray*}

Let $c_2$ be the smallest positive integer greater than  $\left\lceil 3\sqrt{\frac{(r+1)^2}{2}+\frac{2rc_{0}(r^8-r^7+r^4)}{(r^3-2)^2}}\right\rceil$ satisfying that $n+c_{2}$ is a multiple of $r.$

%{\bf Claim:} $n_{r}> \frac{n-(r-1)c_2}{r}.$

If $n_{r}\leqslant  \frac{n-(r-1)c_2}{r},$
then by Lemma~\ref{difference of two type of partitions}, $q(K_{n_{1},\cdots,n_{r}})\leqslant  q\left(K_{\frac{n-(r-1)c_2}{r},\frac{n+c_2}{r},\cdots,\frac{n+c_2}{r}}\right).$
Let $W_{i}$ be the partite set with $|W_{1}|=\frac{n-(r-1)c_2}{r}$ and $|W_{i}|=\frac{n+c_2}{r}$ for $2\leqslant  i\leqslant  r,$  $\pi=\{W_{1}, W_{2}\cup \cdots\cup W_{r}\}$ be a partition  consisting of all partite sets of  $K_{\frac{n-(r-1)c_2}{r},\frac{n+c_2}{r},\cdots,\frac{n+c_2}{r}}.$
The quotient matrix of $Q\left(K_{\frac{n-(r-1)c_2}{r},\frac{n+c_2}{r},\cdots,\frac{n+c_2}{r}}\right)$ corresponding to the partition $\pi$ is
$$Q\left(K_{\frac{n-(r-1)c_2}{r},\frac{n+c_2}{r},\cdots,\frac{n+c_2}{r}}\right)/\pi=\begin{bmatrix}
    \frac{(n+c_2)(r-1)}{r}&\frac{(n+c_2)(r-1)}{r}\\
     \frac{n-(r-1)c_2}{r}&n+\frac{(r-3)(n+c_2)}{r}
\end{bmatrix}.$$
By the equitable partition, we have  $q\left(K_{\frac{n-(r-1)c_2}{r},\frac{n+c_2}{r},\cdots,\frac{n+c_2}{r}}\right)=\rho\left(Q\left(K_{\frac{n-(r-1)c_2}{r},\frac{n+c_2}{r},\cdots,\frac{n+c_2}{r}}\right)/\pi\right).$ Thus,
\begin{eqnarray*}
q\left(K_{\frac{n-(r-1)c_2}{r},\frac{n+c_2}{r},\cdots,\frac{n+c_2}{r}}\right)&=& \frac{n+\frac{2(r-2)(n+c_2)}{r}+\sqrt{(n+\frac{2(r-2)(n+c_2)}{r})^2-4\frac{2(r-1)(r-2)(n+c_2)^2}{r^2}}}{2}\\
&=&\frac{(3r-4)n+2c_2(r-2)+\sqrt{(rn-2(r-2)c_2)^2-8(r-1)(r-2)c_2^2}}{2r}\\
&<& \frac{(3r-4)n+2c_2(r-2)+(rn-2(r-2)c_2)-\frac{4(r-1)(r-2)c_2^2}{rn-2(r-2)c_2}}{2r}\\
&=& 2\left(1-\frac{1}{r}\right)n-\frac{2(r-1)(r-2)c_2^2}{r^2n-2(r-2)rc_2}.
\end{eqnarray*}

Assume that $n=kr+t,$ where $0\leqslant  t <r$. By the result in \cite{Cai Fan 2009}, we have $$q(T_{r}(n))=\frac{(3r-4)k+3t-2+\sqrt{r^2k^2+(2(t+2)r-8t)k+(t-2)^2}}{2}.$$

If $t=0,$ then $q(T_{r}(n))=2\left(1-\frac{1}{r}\right)n>2\left(1-\frac{1}{r}\right)n-\frac{(r+1)^2}{n-2}.$
If $t\neq 0,$ then $q(T_{r}(n))>\frac{(3r-4)k+3t-2+n+2-\frac{4t}{r}-\frac{2(r+1)^2}{n-2}}{2}=2\left(1-\frac{1}{r}\right)n-\frac{(r+1)^2}{n-2}$
as $n$ is sufficiently large.
Since $T_{r}(n)\in \mathcal{G}_{n}$, we have
\begin{eqnarray*}
q(T_{r}(n))&\leqslant & q(\mathcal{G}_{n})=q(G)
< q\left(K_{\frac{n-(r-1)c_2}{r},\frac{n+c_2}{r},\cdots,\frac{n+c_2}{r}}\right)+\frac{4rc_{0}\frac{r^8-r^7+r^4}{(r^3-1-\frac{2c_{0}r^4}{n})^2}}{n}\\
&<& 2\left(1-\frac{1}{r}\right)n-\frac{2(r-1)(r-2)c_2^2}{r^2n-2(r-2)rc_2}+\frac{4rc_{0}\frac{r^8-r^7+r^4}{(r^3-1-\frac{2c_{0}r^4}{n})^2}}{n}\\
&<& q(T_{r}(n))+\frac{(r+1)^2}{n-2}-\frac{2(r-1)(r-2)c_2^2}{r^2n-2(r-2)rc_2}+\frac{4rc_{0}\frac{r^8-r^7+r^4}{(r^3-1-\frac{2c_{0}r^4}{n})^2}}{n}\\
&<& q(T_{r}(n))+\frac{(r+1)^2}{n-2}-\frac{\frac{2}{9}c_2^2}{n-2}+\frac{4rc_{0}\frac{r^8-r^7+r^4}{(r^3-1-\frac{2c_{0}r^4}{n})^2}}{n}\\
&\leqslant  & q(T_{r}(n))+\frac{(r+1)^2-\frac{2}{9}c_2^2+4rc_{0}\frac{r^8-r^7+r^4}{(r^3-1-\frac{2c_{0}r^4}{n})^2}}{n-2}
<  q(T_{r}(n)),
\end{eqnarray*}
where the last inequality holds as $n$ is sufficiently large and $c_2>\left\lceil 3\sqrt{\frac{(r+1)^2}{2}+\frac{2rc_{0}(r^8-r^7+r^4)}{(r^3-2)^2}}\right\rceil.$ This is a contradiction. Thus,  $n_{r}> \frac{n-(r-1)c_2}{r}.$

If $n_{1} >\frac{n+(r-1)^2c_2}{r},$
then by Lemma~\ref{difference of two type of partitions},
$$q(K_{n_{1},\cdots,n_{r}})<  q\left(K_{\frac{n+(r-1)^2c_2}{r},\frac{n-(r-1)c_2}{r},\cdots,\frac{n-(r-1)c_2}{r}}\right)< q\left(K_{\frac{n-(r-1)c_2}{r},\frac{n+c_2}{r},\cdots,\frac{n+c_2}{r}}\right).$$
Using the above discussion, we deduce a contradiction.  Thus,  $n_{1} \leqslant \frac{n+(r-1)^2c_2}{r}.$
\end{proof}

\begin{lemma}\label{the lower bound of minimum component}
Let %$G\in \mathcal{G}_{n},$  $q(G)=q(\mathcal{G}_{n}),$
$\mathbf{x}$ be the  Perron vector of $G$ corresponding to $q(G).$
 If $n$ is  sufficiently large,
 then there exists a positive constant $c_{3}$ satisfying that
$x_{w}> 1-\frac{c_{3}}{n}$ for any $w\in V(G).$
\end{lemma}
\begin{proof}
%Firstly, we show that $x_{u}\geqslant  1-\frac{c_{3}}{n}$ for any $u\in V(G).$
Since $T_{r}(n)\in \mathcal{G}_{n},$ we have        $q(G)=q(\mathcal{G}_{n})\geqslant  q(T_{r}(n))> 2(1-\frac{1}{r})n-\frac{(r+1)^2}{n-2}>n$ as $r\geqslant3.$
By %Theorem~\ref{jin-main lemma} and
Lemma~\ref{the size of each partition},
%$G=G_{0}\cup G_{in},$ where $G_{0}\subseteq K_{n_1,\cdots,n_{r}},$  $G_{in}\subseteq K_{n_1,\cdots,n_{r}}^{c}$ and
%$n_{1}\geqslant \cdots \geqslant n_{r}\geqslant \frac{n-(r-1)c_{2}}{r}.$
$\frac{n+(r-1)^2c_2}{r}\geqslant n_{1} \geqslant \cdots \geqslant  n_{r}>\frac{n-(r-1)c_2}{r}.$
Let $x_{v}=\max_{w\in V(G)}\{x_{w}\}=1$ and $x_{u}=\min_{w\in V(G)}\{x_{w}\}.$
We divide it into two cases.

$\mathbf{Case~1}:$ $u,v\in V_{i}$ for some $i\in [r].$

By $q(G)x_{w}=(Q(G)x)_{w},$ we have
\begin{equation}
\left\{
\begin{array}{ll}
q(G)x_{v}=d(v)x_{v}+\sum\limits_{\substack{w\sim v \\w\in V_{i}} }x_{w}+\sum\limits_{\substack{w\sim v \\w\notin V_{i}} }x_{w} \leqslant  d(v)x_{v}+\sum\limits_{\substack{w\sim v\\w\in V_{i}}}x_{w}+\sum\limits_{w\notin V_{i}}x_{w} \nonumber\\
q(G)x_{u}=d(u)x_{u}+\sum\limits_{\substack{w\sim u \\w\in V_{i}} }x_{w}+\sum\limits_{\substack{w\sim u \\w\notin V_{i}} }x_{w} \geqslant  d(u)x_{u}+\sum\limits_{w\notin V_{i}}x_{w}-\sum\limits_{\substack{w\nsim u \\ w\notin V_{i}}}x_{w},
\end{array}
\right.
\end{equation}
which implies that
\begin{eqnarray*}
(q(G)-d(u))(x_{v}-x_{u})&\leqslant &(d(v)-d(u))x_{v}+\sum\limits_{\substack{ w\sim v\\ w\in V_{i}}}x_{w}+\sum\limits_{\substack{w\nsim u\\ w\notin V_{i}}}x_{w} \\&\leqslant & d(v)-d(u)+c_{0}+\sum\limits_{\substack{w\nsim u\\ w\notin V_{i}}}x_{w} \leqslant  d(v)-d(u)+c_{0}+d_{G_{out}}(u)
\\&\leqslant & d_{V_{i}}(v)+2d_{G_{out}}(u)+c_{0}\leqslant    (4r-2)c_{0},
\end{eqnarray*}
where the last second inequality holds as $d(v)\leqslant  d_{V_{i}}(v)+n-|V_{i}|$ and    $d(u)\geqslant n-|V_{i}|-d_{G_{out}}(u),$  the last inequality holds as $d_{G_{out}}(u)\leqslant  2(r-1)c_{0}$ by Corollary~\ref{property of edge maximal}.
Thus, by $r\geqslant 3$ and $n$ is sufficiently large, there exists a positive constant  $c_{3}'$ such that $x_{u}\geqslant  1-\frac{(4r-2)c_{0}}{q(G)-d(u)}\geqslant  1-\frac{c_{3}'}{q(G)} > 1-\frac{c_{3}'}{n}.$

$\mathbf{Case~2}:$ $u\in V_{i}, v\in V_{j}$ for some $i\neq j\in [r].$
Similarly, we have
\begin{equation}
\left\{
\begin{array}{ll}
q(G)x_{v} \leqslant  d(v)x_{v}+\sum\limits_{\substack{w\sim v \\w\in V_{j}}}x_{w}+\sum\limits_{w\in V_{i}}x_{w}+\sum\limits_{w\notin V_{i}\cup V_{j}}x_{w} \nonumber\\
q(G)x_{u} \geqslant  d(u)x_{u}+\sum\limits_{w\in V_{j}}x_{w}+\sum\limits_{w\notin V_{i}\cup V_{j}}x_{w}-\sum\limits_{\substack{w\nsim u\\w\notin V_{i}}}x_{w},
\end{array}
\right.
\end{equation}
which implies that
\begin{eqnarray*}
(q(G)-d(u))(x_{v}-x_{u})&\leqslant &(d(v)-d(u))x_{v}+\sum\limits_{\substack{ w\sim v\\ w\in V_{j}}}x_{w}+\sum_{w\in V_{i}}x_{w} -\sum_{w\in V_{j}}x_{w}+\sum\limits_{\substack{w\nsim u\\w\notin V_{i}}}x_{w}\\
&\leqslant & (d(v)-d(u))x_{v}+d_{V_{j}}(v)x_{v}+|V_{i}|x_{v}-|V_{j}|x_{u}+d_{G_{out}}(u)x_{v}\\
&\leqslant &  2(d_{V_{j}}(v)+d_{G_{out}}(u)+|V_{i}|-|V_{j}|)x_{v}+|V_{j}|(x_{v}-x_{u})\\
&\leqslant &  2c_{0}+4(r-1)c_{0}+2(r-1)c_{2}+\frac{n+(r-1)^2c_{2}}{r}(x_{v}-x_{u}),
\end{eqnarray*}
where the last second inequality holds as $d(v)\leqslant  d_{V_{j}}(v)+n-|V_{j}|$ and    $d(u)\geqslant n-|V_{i}|-d_{G_{out}}(u),$ the last  inequality holds as $|V_{i}|-|V_{j}|\leqslant  (r-1)c_{2},$  $|V_{j}|\leqslant  \frac{n+(r-1)^2c_{2}}{r}$ and $d_{G_{out}}(u)\leqslant  2(r-1)c_{0}$ by Corollary~\ref{property of edge maximal}.  % by Lemma~\ref{the size of each partition}.

Since $q(G)x_{u} = d(u)x_{u}+\sum\limits_{w\sim u }x_{w}\geqslant 2d(u)x_{u},$ we have $q(G)\geqslant 2d(u).$
If $d(u)\geqslant \frac{3}{4}n,$ by $r\geqslant 3$ and $n$ is sufficiently large, then $$q(G)> d(u)+\frac{n+(r-1)^2c_{2}}{r}.$$  If $d(u)< \frac{3}{4}n,$  by $q(G)>  2(1-\frac{1}{r})n-\frac{(r+1)^2}{n-2},$ $r\geqslant 3$ and $n$ is sufficiently large, then $q(G)> d(u)+\frac{n+(r-1)^2c_{2}}{r}.$
Thus, by $r\geqslant 3$ and $n$ is sufficiently large, there exists a positive constant   $c_{3}''$ such that $$x_{u}\geqslant  1-\frac{2c_{0}+4(r-1)c_{0}+2(r-1)c_{2}}{q(G)-d(u)-\frac{n+(r-1)^2c_{2}}{r}}\geqslant  1-\frac{c_{3}''}{q(G)} > 1-\frac{c_{3}''}{n}.$$

Let $c_{3}=\max\{c_{3}',c_{3}''\}.$ Then $x_{u}> 1-\frac{c_{3}}{n}.$
%Thus, ? $x_{u}\geqslant  1-\frac{1}{\sqrt{n}}\geqslant 1-\frac{c_{1}}{q(G)}.$
%\textcolor{purple}{Need to add this part.} Now we prove $e(G_{out})\leqslant  c_{2}.$
%By way of contradiction. If $e(G_{out})> c_{2}.$
%Since $G=G_{0}\cup G_{in},$ we have
%\begin{eqnarray*}
%q(T_{r}(n))&\leqslant & q(G)=x^{T}Q(G)x=\sum_{ij\in E(G)}(x_{i}+x_{j})^2\\
%&=&\sum_{ij\in E(K_{n_{1},\cdots,n_{r}})}(x_{i}+x_{j})^2+\sum_{ij\in E(G_{in})}(x_{i}+x_{j})^2-\sum_{ij\in E(G_{out})}(x_{i}+x_{j})^2\\
%&\leqslant & q(T_{r}(n))+4e(G_{in})-4e(G_{out})x_{min}^2\\
%&\leqslant & q(T_{r}(n))+4e(G_{in})-4e(G_{out})\left(1-\frac{c_{1}}{n}\right)^2\\
%&\leqslant & q(T_{r}(n))+4rc_{0}-4e(G_{out})\left(1-\frac{c_{1}}{n}\right)^2 < q(T_{r}(n)),
%\end{eqnarray*}
%where $e(G_{in})=\sum_{i=1}^{r}e(G[V_{i}])\leqslant  rc_{0}.$
\end{proof}

\begin{lemma}\label{turan graph}
%Let %$G\in \mathcal{G}_{n},$  $q(G)=q(\mathcal{G}_{n})$ and
%$r\geqslant 3.$
%where $G=G_{0}\cup G_{1}$ and $G_{0}\subseteq K_{n_1,\cdots,n_{r}}.$
 If $n$ is  sufficiently large,
 then $|n_{i}-n_{j}|\leqslant  1$
for any $1\leqslant  i< j\leqslant  r.$
\end{lemma}
\begin{proof}
We prove it by way of contradiction. Then there exist two integers $1\leqslant  i,j\leqslant  r$ such that $|n_{i}-n_{j}|\geqslant 2.$
%By Theorem~\ref{jin-main lemma}, we have $G=G_{0}\cup G_{in},$ where $G_{0}\subseteq K_{n_1,\cdots,n_{r}}$ and $G_{in}\subseteq K_{n_1,\cdots,n_{r}}^{c}.$

%By (ii) in Lemma~\ref{difference of two type of partitions}, we have $q(T_{r}(n))> q(K)+\frac{2(r-2)}{r^2n}.$

Let $\mathbf{x}$ be the  Perron vector of $T_{r}(n)$ corresponding to $q(T_{r}(n))$ and $a=n-r\lfloor\frac{n}{r}
\rfloor.$
If $a=0,$ then $T_{r}(n)$ is a regular graph and $x_{1}=\cdots=x_{n}=1.$
If $a\neq 0,$
%since  $T_{r}(n)$ is a complete $r$-partite graph on $n$ vertices where each partite set has either $\lfloor\frac{n}{r} \rfloor$ or $\lceil\frac{n}{r} \rceil$ vertices,
then we assume
$\mathbf{x}=(\underbrace{x_{1},\cdots,x_{1}}_
{a\left\lceil{\frac{n}{r}}\right\rceil},
\underbrace{x_{2},\cdots,x_{2}}_{n-a\left\lceil{\frac{n}{r}}
\right\rceil})^{\mathrm{T}}.$
Thus
$$q(T_{r}(n))x_{1}=(n-\lceil\frac{n}{r}\rceil)x_{1}+(a-1)\lceil\frac{n}{r}\rceil x_{1}+(n-a\lceil\frac{n}{r}\rceil)x_{2}$$
and
$$q(T_{r}(n))x_{2}=(n-\lfloor\frac{n}{r}\rfloor)x_{2}+a\lceil\frac{n}{r}\rceil x_{1}+(n-a\lceil\frac{n}{r}\rceil-\lfloor\frac{n}{r}\rfloor)x_{2},$$
which implies that
$$\left(q(T_{r}(n))-n+2\lceil\frac{n}{r}\rceil\right)x_{1}=\left(q(T_{r}(n))-n+2\lfloor\frac{n}{r}\rfloor\right)x_{2}.$$
%Without loss of generality, we assume that $x_{2}=1.$
Then
$$x_{2}\geqslant x_{1}= \frac{q(T_{r}(n))-n+2\lfloor\frac{n}{r}\rfloor}{q(T_{r}(n))-n+2\lceil\frac{n}{r}\rceil}= 1-\frac{2}{q(T_{r}(n))-n+2\lceil\frac{n}{r}\rceil}.$$
Combining with $q(T_{r}(n))\geqslant  2\delta(T_{r}(n))\geqslant  2(n-\lceil\frac{n}{r}\rceil),$ we have
$x_{1}\geqslant 1-\frac{2}{n}.$

On the one hand, denote by $H$ is an $F$-free graph of order $n$ with $e(H)=ex(n,F).$
By the proof of Corollary~\ref{corollary}, we have $H=H_{0}\cup H_{in},$ $H_{0}\subseteq T_{r}(n),$ $H_{in}\subseteq T_{r}^{c}(n),$ $e(H_{in})-e( H_{out})=c_{0}$ and $\delta(H)\geqslant  \lfloor(1-\frac{1}{r})n\rfloor.$ Then  $H\in \mathcal{G}_{n}$   as  $\delta(H)\geqslant \lfloor(1-\frac{1}{r})n\rfloor =  n-\lceil\frac{n}{r}\rceil$.  Thus,
\begin{eqnarray}\label{upper bound1}
&&q\left(G\right)\geqslant q\left(H\right)\geqslant\frac{\mathbf{x}^{T}Q\left(H\right)\mathbf{x}}{\mathbf{x}^{T}\mathbf{x}} \nonumber
\\
&=& \frac{\mathbf{x}^{T}Q\left(T_{r}(n)\right)\mathbf{x}}{\mathbf{x}^{T}\mathbf{x}}+\frac{\sum_{ij\in E\left(H_{in}\right)}\left(x_{i}+x_{j}\right)^{2}}{\mathbf{x}^{T}\mathbf{x}}-\frac{\sum_{ij\in E\left(H_{out}\right)}\left(x_{i}+x_{j}\right)^{2}}{\mathbf{x}^{T}\mathbf{x}}\nonumber\\
&=&q\left(T_{r}(n)\right)+\frac{\sum_{ij\in E\left(H_{in}\right)}\left(x_{i}+x_{j}\right)^{2}}{\mathbf{x}^{T}\mathbf{x}}-\frac{\sum_{ij\in E\left(H_{out}\right)}\left(x_{i}+x_{j}\right)^{2}}{\mathbf{x}^{T}\mathbf{x}}\nonumber\\
&\geqslant& q\left(T_{r}(n)\right)+\frac{4e(H_{in})}{\mathbf{x}^{T}\mathbf{x}}x_{min}^{2}-\frac{4e(H_{out})}{\mathbf{x}^{T}\mathbf{x}}\geqslant q\left(T_{r}(n)\right)+\frac{4e(H_{in})}{\mathbf{x}^{T}\mathbf{x}}\left(1-\frac{2}{n}\right)^2-\frac{4e(H_{out})}{\mathbf{x}^{T}\mathbf{x}}\nonumber\\
&\geqslant& q\left(T_{r}(n)\right)+\frac{4e(H_{in})}{\mathbf{x}^{T}\mathbf{x}}\left(1-\frac{4}{n}\right)-\frac{4e(H_{out})}{\mathbf{x}^{T}\mathbf{x}}\geqslant q\left(T_{r}(n)\right)+\frac{4e(H_{in})}{n}\left(1-\frac{4}{n}\right)-\frac{4e(H_{out})}{n(1-\frac{2}{n})^2}\nonumber\\
&\geqslant&  q\left(T_{r}(n)\right)+\frac{4(e(H_{in})-e(H_{out}))}{n}+\frac{4e(H_{out})}{n}-\frac{16e(H_{in})}{n^2}-\frac{4e(H_{out})}{n-4} \nonumber\\
&=&
q\left(T_{r}(n)\right)+\frac{4c_{0}}{n}-\frac{16e(H_{out})}{n(n-4)}-\frac{16(e(H_{out})+c_{0})}{n^2}\geqslant
q\left(T_{r}(n)\right)+\frac{4c_{0}}{n}-\frac{16(2e(H_{out})+c_{0})}{n(n-4)}\nonumber\\
&\geqslant & q\left(T_{r}(n)\right)+\frac{4c_{0}}{n}-\frac{16c_{0}(4(r-1)rc_{0}+1)}{n(n-4)},
\end{eqnarray}
where the sixth inequality holds as $x_{min}\geqslant 1-\frac{2}{n}$ %the last third equality holds by Corollary~\ref{corollary}
and the last  inequality holds as $e(G_{out})\leqslant  2(r-1)rc_{0}^2$ by Corollary%~\ref{corollary} and
~\ref{property of edge maximal}.

On the other hand, let $K=K_{n_1,\cdots,n_{r}}$ and  $\mathbf{y}$ be the Perron vector of $G$ corresponding to $q(G).$ %Assume that $\max_{i\in V(G)}\{y_{i}\}=1.$
Since %$G=G_{0}\cup G_{in}$ by Theorem~\ref{jin-main lemma} and
$y_{i}> 1-\frac{c_{3}}{n}$ for all $i\in V(G)$ by Lemma~\ref{the lower bound of minimum component}, we have
\begin{eqnarray}\label{upper bound2}
q(G)&=&\frac{\mathbf{y}^{T}Q(G)\mathbf{y}}{\mathbf{y}^{T}\mathbf{y}}=\frac{\mathbf{y}^{T}Q(K)\mathbf{y}}{\mathbf{y}^{T}\mathbf{y}}+\frac{\sum_{ij\in E(G_{in})}(y_{i}+y_{j})^2}{\mathbf{y}^{T}\mathbf{y}}-\frac{\sum_{ij\in E(G_{out})}(y_{i}+y_{j})^2}{\mathbf{y}^{T}\mathbf{y}} \nonumber\\
&\leqslant  &  \frac{\mathbf{y}^{T}Q(K)\mathbf{y}}{\mathbf{y}^{T}\mathbf{y}}+\frac{4e(G_{in})}{\mathbf{y}^{T}\mathbf{y}}-\frac{4e(G_{out})(1-\frac{c_{3}}{n})^2}{\mathbf{y}^{T}\mathbf{y}}\nonumber\\
%&\leqslant  &  q(K)+\frac{4e(G_{in})}{\mathbf{y}^{T}\mathbf{y}}-\frac{4e(G_{out})(1-\frac{c_{3}}{n})^2}{\mathbf{y}^{T}\mathbf{y}}
%\nonumber\\
&\leqslant  &  q(K)+\frac{4(e(G_{in})-e(G_{out}))}{\mathbf{y}^{T}\mathbf{y}}+\frac{4e(G_{out})\left(\frac{2c_{3}}{n}-\frac{c_{3}^2}{n^2}\right)}{\mathbf{y}^{T}\mathbf{y}}
\nonumber\\
&< &  q(K)+\frac{4c_{0}}{n-2c_{3}}+\frac{16(r-1)rc_{0}^2c_{3}}{n(n-2c_{3})},
\end{eqnarray}
where the last inequality holds as $e(G_{in})-e(G_{out})\leqslant  c_{0},$  $e(G_{out})\leqslant  2(r-1)rc_{0}^2$ and $\mathbf{y}^{T}\mathbf{y}> n\left(1-\frac{c_{3}}{n}\right)^2>n-2c_{3}.$

According to  (\ref{upper bound1}) and (\ref{upper bound2}), we have
\begin{eqnarray}\label{combining upper bound}
q(T_{r}(n))-q(K)&<& \frac{4c_{0}}{n-2c_{3}}+\frac{16(r-1)rc_{0}^2c_{3}}{n(n-2c_{3})}-\frac{4c_{0}}{n}+\frac{16c_{0}(4(r-1)rc_{0}+1)}{n(n-4)}\nonumber\\
&= & \frac{8c_{0}c_{3}}{n(n-2c_{3})}+\frac{16(r-1)rc_{0}^2c_{3}}{n(n-2c_{3})}+\frac{16c_{0}(4(r-1)rc_{0}+1)}{n(n-4)} \leqslant  \frac{c_{4}}{n^2},
\end{eqnarray}
where $c_{4}$ is a positive constant  and the last inequality holds as   $n$ is sufficiently large.

By Lemma~\ref{difference of two type of partitions}, we have $q(T_{r}(n))> q(K)+\frac{2(r-2)}{r^2n}.$
Combining with (\ref{combining upper bound}), we have
$$\frac{2(r-2)}{r^2n}< q(T_{r}(n))-q(K)< \frac{c_{4}}{n^2},$$
a contradiction as $n$ is sufficiently large. We complete the proof.
\end{proof}
%\textcolor{red}{The next lemma is the key to the Main theorem.}
\begin{lemma}\label{edge extremal graph}
%Let $G\in \mathcal{G}_{n}$ and $q(G)=q(\mathcal{G}_{n}).$
%where $G=G_{0}\cup G_{1}$ and $G_{0}\subseteq K_{n_1,\cdots,n_{r}}.$
If $n$ is  sufficiently large,
 then $e(G)=ex(n,F).$
\end{lemma}
\begin{proof}
We prove it by way of contradiction. Suppose that $e(G)\leqslant  ex(n,F)-1.$ Let $H\in Ex(n,F)$ with $V(H)=V(G).$ Then $e(H)=ex(n,F)=t_{r}(n)+c_{0}$ by Remark~\ref{c0}.
 According to Corollaries~\ref{corollary} and \ref{property of edge maximal}, we have $H=H_{0}\cup H_{in}$ and $\delta(H)\geqslant  \lfloor(1-\frac{1}{r})n\rfloor,$ where  $H_{0}\subseteq T_{r}(n)$ and $H_{in}\subseteq T_{r}^{c}(n).$
Hence, $H\in \mathcal{G}_{n}.$ In the following, we  prove that $q(H)> q(G).$
By Lemma~\ref{turan graph}, $K_{n_1,\cdots,n_{r}}=T_{r}(n).$
%Let $V_{1},\cdots,V_{r}$ be a partition of $G.$

Let $E_{1}=E(G)\setminus E(H)$ and $E_{2}=E(H)\setminus E(G).$ Then $E(H)=(E(G)\cup E_{2})\setminus E_{1}$ and $|E_{2}|\geqslant |E_{1}|+1.$ By Corollaries~\ref{corollary} and \ref{property of edge maximal}, we have
\begin{eqnarray}\label{the size of E1}
|E_{1}|\leqslant  \sum_{i=1}^{r}e(V_{i})+ e(H_{out})\leqslant  rc_{0}+2(r-1)rc_{0}^2.%\binom{}{}\sum_{1\leqslant  i<j\leqslant  r} |A_{i}||A_{j}|.
\end{eqnarray}
Let $\mathbf{x}$ be the  Perron vector of $G$ corresponding to $q(G).$ Then for sufficiently large $n,$ we have
\begin{eqnarray*}
q(H)&\geqslant& \frac{\mathbf{x}^{T}Q(H)\mathbf{x}}{\mathbf{x}^{T}\mathbf{x}}= \frac{\mathbf{x}^{T}Q(G)\mathbf{x}}{\mathbf{x}^{T}\mathbf{x}}+\frac{\sum_{ij\in E_{2}}(x_{i}+x_{j})^2}{\mathbf{x}^{T}\mathbf{x}}-\frac{\sum_{ij\in E_{1}}(x_{i}+x_{j})^2}{\mathbf{x}^{T}\mathbf{x}}\nonumber\\
&\geqslant & q(G)+\frac{4}{\mathbf{x}^{T}\mathbf{x}}(|E_{2}|x_{min}^2- |E_{1}|)\geqslant  q(G)+\frac{4}{\mathbf{x}^{T}\mathbf{x}}((|E_{1}|+1)x_{min}^2- |E_{1}|)\nonumber\\
&= & q(G)+\frac{4}{\mathbf{x}^{T}\mathbf{x}}(|E_{1}|+1)\left(x_{min}^2-1+ \frac{1}{|E_{1}|+1}\right)> q(G),
\end{eqnarray*}
where the last inequality holds by   (\ref{the size of E1}), $x_{min}\geqslant  1-\frac{c_{3}}{n}$ and $n$ is sufficiently large,
a contradiction to the assumption that $G$ has the maximum signless Laplacian spectral radius among $\mathcal{G}_{n}$. Therefore, $e(G)=ex(n,F).$
\end{proof}

Now we are ready to give the proof of Theorem~\ref{main result}.

\begin{proofof}~{\bf\ref{main result}}.
Let $G$ be a graph  with the maximum signless Laplacian spectral radius among all  $F$-free graphs of order $n$.
Since $ex(n,F)=t_{r}(n)+O(1),$ we have
$\pi(F)=1-\frac{1}{r}.$  Thus,  for sufficiently large $n,$ we have
$$\left|ex(n,F)-ex(n-1,F)-\pi(F)n\right|=\left|\left(1-\frac{1}{r}\right)\frac{n^2}{2}-\left(1-\frac{1}{r}\right)\frac{(n-1)^2}{2}-\left(1-\frac{1}{r}\right)n+O(1)\right|\leqslant  \sigma n,$$
which implies that the condition (\ref{the first inequality}) of Theorem~\ref{theorem of liyongtao} holds.
%Let $\mathcal{G}_{n}$ be the collection of all $n$-vertex $F$-free graphs $H$ with minimum degree $\delta(H)\geqslant  \left(\pi(F)-\epsilon\right)n.$
According to Lemma~\ref{turan graph}, we have $G$ is  a spanning subgraph of the graph obtained from $T_{r}(n)$ by adding at most $rc_{0}$ edges.
By Lemma~\ref{the upper bound of signless laplacian spectral radius} recursively,
$$q(\mathcal{G}_{n})=q(G)\leq q(T_{r}(n))+\frac{4rc_{0}(n-\lfloor \frac{n}{r}\rfloor)}{\left(\lfloor \frac{n}{r}\rfloor-2c_{0}\right)^2}.$$
Combining with  $q(\mathcal{G}_{n})\geqslant  q(T_{r}(n))$,  we have $q(\mathcal{G}_{n})=q(T_{r}(n))+o(1)=2\left(1-\frac{1}{r}\right)n+o(1).$ Then  for sufficiently large $n,$ we have
$$\left|q(\mathcal{G}_{n})-4ex(n,F)n^{-1}\right|=\left|q(T_{r}(n))-4ex(n,F)n^{-1}+o(1)\right|=o(1)\leqslant  \sigma,$$
which implies that the condition (\ref{the second inequality}) of Theorem~\ref{theorem of liyongtao} holds.
Thus, by Theorem~\ref{theorem of liyongtao}, we have $q(G)=q(\mathcal{G}_{n})$ and $G\in \mathcal{G}_{n}$ for sufficiently large $n.$
According to Lemma~\ref{edge extremal graph}, we have  $e(G)=ex(n,F),$ which implies that $G\in Ex(n,F).$
We complete the proof.
\end{proofof}

\begin{remark}\label{a solution to problem} For $k\geqslant 1,$  $t\geqslant 2$ and $n\geqslant 16k^3t^{8},$
Chen et al.%~ Gould, Pfender and Wei
~\cite{Chen Gould Pfender Wei 2003} proved an exact result for  $ex(n,F_{k,t})$  and also determined  the extremal graph in $Ex(n, F_{k,t}),$ which is obtained from $T_{t-1}(n)$ by adding $O(1)$ edges. While for $k\geqslant 1$ and $t\geqslant 3,$ Desai et al.~\cite{DKLNTW2022} proposed that $Ex_{ssp}(n,F_{k,t})=\{K_{k(t-2)}\vee \overline{K}_{n-k(t-2)}\}$ if $n$ is sufficiently large.
According to  Theorem~\ref{main result}, we provide a negative answer to this  problem. \end{remark}

\end{document}